\documentclass[12pt,psamsfonts]{amsart}
\usepackage{amsmath}
\usepackage{amsthm}
\usepackage{amssymb}
\usepackage{amscd}
\usepackage{amsfonts}
\usepackage{amsbsy}

\def\be{\begin{equation}}
\def\ee{\end{equation}}

\def\bq{\begin{eqnarray}}
\def\eq{\end{eqnarray}}

\def\beq{\begin{eqnarray*}}
\def\eeq{\end{eqnarray*}}

\newcommand{\eps}{\varepsilon}
\newcommand{\la}{\lambda}
\newcommand{\ov}{\overline{v}}

\newcommand{\e}{\varepsilon}

\newtheorem {theorem} {Theorem}%[section]

\newtheorem {definition} [theorem]{Definition}

\newtheorem {lemma}  [theorem]{Lemma}

\newcommand{\R}{\mathbb{R}}

\begin{document}

\title[Essential perturbations]
{Essential perturbations of polynomial vector fields with a period annulus}

\author[A. Buic\u{a}, J. Gin\'e and M. Grau]
{Adriana Buic\u{a}$^1$, Jaume Gin\'e$^2$ and Maite Grau$^2$}

\address{$^1$ Departamentul de Matematic\u{a}, Universitatea Babe\c{s}--Bolyai,
 Str. Kog\u{a}lniceanu 1, 400084 Cluj--Napoca, Romania}

\email{abuica@math.ubbcluj.ro}

\address{$^2$ Departament de Matem\`atica, Universitat de Lleida,
Avda. Jaume II, 69; 25001 Lleida, Spain}

\email{gine@matematica.udl.cat, mtgrau@matematica.udl.cat}


\date{}
\dedicatory{}

\maketitle

\begin{abstract}
In this paper we first give the explicit definition of essential perturbation.
Secondly, given a perturbation of a particular family of centers of polynomial differential systems of arbitrary degree for which we explicitly know its Poincar\'e--Liapunov constants, we give the structure of its $k$-th Melnikov function.
This result generalizes the result obtained by Chicone and Jacobs for perturbations of degree at most two of any center of a quadratic polynomial system. Moreover we study the essential perturbations
for all the centers of the differential systems
\[ \dot{x} \, = \, -y + P_{\rm d}(x,y), \quad \dot{y} \, = \, x +
Q_{\rm d}(x,y), \] where $P_{\rm d}$ and $Q_{\rm d}$ are
homogeneous polynomials of degree ${\rm d}$,
for ${\rm d}=2$ and ${\rm d}=3$.
\end{abstract}

\section{Introduction \label{sect1}}

One of the last open problems from the list suggested by Hilbert at
the beginning of the 20th century is the 16th problem, see \cite{H}.
The second part of this problem focus on the study of the limit
cycles of  planar polynomial real differential systems.
More specifically Hilbert's 16th problem (part b) is the following:

\smallskip

{\it For the family of polynomial differential systems  of degree
${\rm d}$, is there a uniform upper bound, depending only on ${\rm
d}$, for the number of limit cycles of each system in the family?}

\smallskip

 This problem is still unsolved even for quadratic systems,
i.e. for the case ${\rm d}=2$ (see \cite{DumRouRou}). Moreover, a
weaker version is included in Smale's list of problems to be solved
for the 21st century (see \cite{Sm}).

Roussarie establishes in \cite{Ro} that this global
problem can be reduced to several local bifurcation problems.  In
fact, finite cyclicity of any limit periodic set, in terms of the
degree of the system, implies the solution to Hilbert's 16th
problem. The cyclicity problem has been studied by several authors
also in its relation with the center problem, see for instance
\cite{CauDum,CauGas,Chri,GasGine,GasTor,Gavr,IlyYak}. Roughly speaking, the cyclicity of a
limit periodic set of a polynomial system of degree at most ${\rm
d}$, is the maximum number of limit cycles that can bifurcate from
the given limit periodic set inside the family of all polynomial
systems of degree ${\rm d}$; see Definition 12 in \cite{Ro} for a precise definition. Some usual examples of limit periodic
sets are: a weak focus, a center point, a period annulus, a
homoclinic loop, a heteroclinic graphic.

In this work we contribute to the study of the cyclicity of a period annulus $\mathcal{P}$ surrounding a nondegenerate center point. In order to state our contribution we introduce some notations.

Consider a (fixed) system with a nondegenerate center at the origin:
\begin{equation}\label{eq1}
\dot{x} \, = \, - y + P(x,y), \quad \dot{y} \, = \, x + Q(x,y),
\end{equation} where  $P$ and $Q$ are real polynomials of degree at most ${\rm d}$ without constant nor linear
terms.
For system (\ref{eq1}), there exists an analytical first integral
$H(x,y)$ and an inverse integrating factor $V(x,y)$ with $V(0,0)=1$,
see \cite{Gine,Po,Reeb}. The periodic orbits $\gamma_h \subset
\{H=h\}$ surrounding the origin of \eqref{eq1} can be parameterized
by the values of $H$. The period annulus $\mathcal{P}$ is defined by
\[ \displaystyle \mathcal{P} \, = \, \{ \gamma_h\, : \, {h \in (h_0,h_1)} \}, \]
where $h_0 \in \mathbb{R}$ corresponds to the inner boundary (i.e.
the origin) and $h_1 \in \mathbb{R} \cup \{+\infty\}$ corresponds to
the outer boundary. Consider now a family of perturbations of (\ref{eq1}):
\begin{equation}
\begin{array}{lll}
\displaystyle  \dot{x} &  = & \displaystyle  - y + P(x,y) + \,
\varepsilon \, p(x,y, \tilde\lambda,\varepsilon), \vspace{0.2cm} \\
\dot{y} \displaystyle & = & \displaystyle x + Q(x,y) + \,
\varepsilon \, q(x,y, \tilde\lambda,\varepsilon),
\end{array} \label{eq2}
\end{equation}
where $p$ and $q$ are polynomials in $x,y$ of degree ${\rm d}$ and
analytic functions in the small {\it bifurcation parameter}
$\varepsilon$ and in the parameters $\tilde\lambda\in\mathbb{R}^m$.
We remind that we consider the problem of bifurcation of limit
cycles from the period annulus $\mathcal{P}$ of system \eqref{eq1}
in the family \eqref{eq2}. The two mostly used methods to solve this problem are the averaging method, see for instance \cite{BLV}, and the Melnikov functions, see for instance \cite{Ro, Rou}. In this paper we mainly deal with the computation of the so called {\it Melnikov functions}. However,
Melnikov functions cannot always be explicitly computed. In order to
define what is a Melnikov function,  we  consider the Poincar\'e map
$\ \pi(\cdot;\varepsilon): \Sigma \to \Sigma$ associated to system
(\ref{eq2}) and the period annulus $\mathcal{P}$, where $\Sigma$ is
a transversal section parameterized by $h$, passing through the
origin and cutting the whole $\mathcal{P}$. We are under the
assumption that for $\varepsilon=0$ system (\ref{eq2}) has a center
at the origin, thus we have that $\pi(h;0) \, = \, h$ for all
$h\in[h_0,h_1)$. By the analyticity of the Poincar\'e map with
respect to parameters, we have the displacement map
$$d(h;\varepsilon)=\pi(h;\varepsilon) \, - \, h \, = \, M_1(h)\,
\varepsilon \, + \, M_2(h)\, \varepsilon^2 \, + \ldots +\, M_r(h)\,
\varepsilon^r \,  +\, \mathcal{O}(\varepsilon^{r+1}).$$ Depending on
the parameters $\tilde \lambda$, there exists some $k\geq 1$ such
that $M_r(h)\equiv 0$ for any $1\leq r< k$ and $M_k(h)\not\equiv 0$,
i.e.
$$d(h;\varepsilon) \, =  \, M_k(h) \, \varepsilon^k \, + \,
\mathcal{O}(\varepsilon^{k+1}).$$ The function $M_k(h)$ is called
the Melnikov function of order $k$. The isolated zeroes of $M_k(h)$
(counted with multiplicity) allow to study  limit cycles of system
(\ref{eq2}) which bifurcate from the orbits of the period annulus of
system (\ref{eq1}) (see, for instance, subsection 4.3.4 of \cite{Ro}). In particular, the following result, which is Theorem 6.1 in \cite{JibinLi}, is well-known.
\begin{theorem} \label{thbifur}
Let $M_k(h)$ be the Melnikov function of order $k$ associated to system {\rm (\ref{eq2})} and let $h^{*} \in (h_0,h_1)$. We denote by $\gamma_h \subset \{H=h\}$ the periodic orbits surrounding the origin of {\rm (\ref{eq1})}. The following statements hold.
\begin{itemize}
\item [{\rm (i)}] If there exists a limit cycle $\Gamma_{\varepsilon,h^{*}}$ of system {\rm (\ref{eq2})} such that $\Gamma_{\varepsilon,h^{*}} \to \gamma_{h^{*}}$ as $\varepsilon \to 0$, then $M_k(h^{*}) \, = \, 0$.
\item[{\rm (ii)}] If $M_k(h^{*}) \, = \, 0$ and $M_k'(h^{*}) \, \neq \, 0$, then there exists a hyperbolic limit cycle $\Gamma_{\varepsilon,h^{*}}$ of system {\rm (\ref{eq2})} such that $\Gamma_{\varepsilon,h^{*}} \to \gamma_{h^{*}}$ as $\varepsilon \to 0$.
\item[{\rm (iii)}] If $M_k^{(i)}(h^{*}) \, = \, 0$ for $i=\overline{0,r-1}$ and $M_k^{(r)}(h^{*}) \, \neq \, 0$ {\rm (}that is, $h^{*}$ is a zero of multiplicity $r$ of $M_k(h)${\rm )}, then {\rm (\ref{eq2})} has at most $r$ limit cycles for $\varepsilon$ sufficiently small in the vicinity of $\gamma_{h^{*}}$.
\item[{\rm (iv)}] The total number of isolated zeros of $M_k(h)$ (taking into account their multiplicity) is an upper bound for the number
of limit cycles of system {\rm (\ref{eq2})} that bifurcate from the periodic orbits of the considered period annulus of {\rm (\ref{eq1})}.
\end{itemize} \end{theorem}

We have defined Melnikov functions in terms of the displacement map. Another way to compute the function $M_1(h)$ is through the following line
integral:
\[ M_1(h) \, = \, \oint_{H(x,y)\, = \, h}
\frac{q(x,y, 0) \, dx \, - \, p(x,y, 0) \, dy}{V(x,y)}\, , \] where
$V(x,y)$ is an inverse integrating factor of system (\ref{eq1})
corresponding to the first integral $H$.   The expression of the
Melnikov function of order $k$ involve, in general, iterated integrals up to order $k$ (see again, for instance, subsection 4.3.4 of \cite{Ro}). Hence, the explicit computation of high order Melnikov functions may become computationally cumbersome or
even impossible. It turns out, however, that the expressions of the Melnikov functions for a particular example, obey some pattern. The aim of this work is to unveil this pattern. More exactly, we show in Theorem \ref{thcj} that the
Melnikov function of order $k$  of system \eqref{eq2} is a (finite)
linear combination of some functions, which we denote by $B_1(h)$, $B_3(h)$, $\ldots$, $B_{2N+1}(h)$. These functions  do not depend
neither on $k$, nor on the parameters $\tilde\lambda$. The
coefficients of the linear combination can be found only if one
knows the expression of the Poincar\'{e}--Liapunov constants for the
family. For certain particular cases, this has already been shown in
\cite{ChiJac, Iliev}. We give in Section \ref{sect2} a general framework to this approach. \\

\noindent Chicone and Jacobs in \cite{ChiJac} and Iliev in \cite{Iliev}
succeeded to find the essential perturbations of quadratic systems
when considering the problem of finding the cyclicity of a period
annulus. We present in the sequel a formal definition of this
notion.
\begin{definition} Given a parametric family of planar polynomial
differential systems {\rm (\ref{eq2})} which unfold a system with a
period annulus $\mathcal{P}$, an {\em essential perturbation} is a
choice of the parameters $\tilde\lambda$ such that:
\begin{itemize}
\item[{\rm (i)}] the number of isolated zeros (counted with multiplicity) of the corresponding Melnikov function is greater or equal to the number of isolated zeros (counted with multiplicity) of the  Melnikov function corresponding to any other value of $\tilde\lambda$;
\item[{\rm (ii)}] the order of the Melnikov function which satisfies {\rm (i)}
 is the lowest possible;
\item[{\rm (iii)}] the number of involved parameters is the lowest
possible satisfying {\rm (i)} and {\rm (ii)}.
\end{itemize}
\end{definition}
In Section \ref{sect2}, we explain how to find the essential perturbations of a parametric family of planar polynomial
differential systems which unfold a system with a period annulus surrounding a nondegenerate center by using Theorem \ref{thcj}.

Section \ref{sect3} contains the description of the essential perturbations
for all the centers of the form
\[ \dot{x} \, = \, -y + P_{\rm d}(x,y), \quad \dot{y} \, = \, x +
Q_{\rm d}(x,y), \] where $P_{\rm d}$ and $Q_{\rm d}$ are homogeneous
polynomials of degree ${\rm d}$, for ${\rm d}=2$ and ${\rm d}=3$. We
remark that the quadratic case was already described in \cite{Iliev}
for quadratic systems written in complex form and we consider
systems written in Bautin form. In Section \ref{sect4} we correct the study of
a particular quadratic system which appears in \cite{BGY}. Last Section \ref{sect5} contains a remark about the finiteness of the number of limit cycles bifurcating from the considered period annulus $\mathcal{P}$.

\section{Essential perturbations. A general framework \label{sect2}}

\noindent First we consider a parametric family of planar polynomial
real systems of the form
\begin{equation} \dot{x} \, = \, -y + \la_1 x +  \ P(x,y, \mathbf{\lambda}), \quad
\dot{y} \, = \, x + \la_1 y +   Q(x,y, \mathbf{\lambda}),
\label{eq3}
\end{equation} where $P$ and $Q$ are polynomials in $x$, $y$ and
$\mathbf{\lambda} \, = \, (\lambda_1, \lambda_2, \ldots, \lambda_n)
\in \mathbb{R}^n$, and the subdegree in $x$ and $y$ of $P$ and $Q$
is at least $2$. We consider a fixed $\mathbf{\la}^* \in
\mathbb{R}^n$ and we assume that for $\mathbf{\lambda} \, = \,
\mathbf{\la}^*$ system (\ref{eq3}) has a center at the origin, whose period annulus is denoted by $\mathcal{P}$. Like
we discussed in the Introduction for systems \eqref{eq1} and
\eqref{eq2}, we consider a section $\Sigma$ through the origin,
transversal to the flow of \eqref{eq3} in the whole period annulus $\mathcal{P}$ when $\lambda$ is in a small
neighborhood of $\lambda^*$, and parameterized by $h$. This time we
assume that  $h=0$ corresponds to the origin of coordinates. We also
consider the displacement map $d(h; \mathbf{\lambda}) \, = \,
\pi(h;\mathbf{\lambda})-h$ associated to family (\ref{eq3}). We
recall that $\pi(h;\mathbf{\lambda})$ denotes the Poincar\'e map,
emphasizing that here it depends on the parameters $\la$.
\newline

The basic idea to tackle the bifurcation of limit cycles from $\mathcal{P}$ is founded on properties of zeros of analytic functions of several variables depending on parameters. We will mainly use the description of these ideas given in subsection 6.1 of the book \cite{RomSha}. The same tools can be found in chapter 4 of the book of Roussarie \cite{Ro}, see also \cite{Rou}. We denote by $v_{2j+1}(\mathbf{\lambda})$, $j=\overline{0,N}$, the Poincar\'e--Liapunov constants associated to the origin of the family of polynomial differential systems (\ref{eq3}). See, for instance, chapter 3 in \cite{RomSha} for their definition. The Poincar\'e--Liapunov constants are the basic tool to solve the center problem, see e.g. \cite{Bautin, CauGas, LliValls, Sch, Sib}. We remark that the Poincar\'e--Liapunov constants are polynomials in $\lambda$ and that their number $N+1$ only depends on the considered family (\ref{eq3}). The following statement corresponds to Lemma 6.1.6 in \cite{RomSha} but written with our notation and our assumptions.
\begin{lemma} \label{technicalemma}
There exist positive numbers $\varepsilon_1$ and $\delta_1$ such that the displacement
map $d(h; \mathbf{\lambda})$ is analytic for $|h|<\varepsilon_1$ and
$\|\la-\la^*\|<\delta_1$ and there exist $N+1$ analytic functions $b_{2j+1}(h,\lambda)$, $j=\overline{0,N}$, with $b_{2j+1}(0,\mathbf{\la}^{*})$ a nonzero constant, such that
\begin{equation} \label{eqdq} d(h;\mathbf{\lambda}) \, = \, \sum_{j=0}^{N} v_{2j+1}(\mathbf{\lambda})\,  h^{2j+1} \, b_{2j+1}(h,\mathbf{\lambda}) \end{equation} holds in the set
$\{ (h,\lambda) \, : \, |h|<\varepsilon_1 \ \mbox{and} \ \|\lambda-\lambda^{*}\|<\delta_1\}$.
\end{lemma}
In addition, and due to the structure of the Poincar\'e--Liapunov constants, it is known that $v_{1}(\la) \, = \, \la_1$, $b_1(0,\la) \, = \,
(e^{2\pi \la_1}-1)/\la_1$, and, for each $j>0$, $v_{2j+1}(\la)$ and
$b_{2j+1}(h,\mathbf{\la})$ do not depend on $\la_1$. We remark that the development of $d(h,\lambda)$ in powers of $h$ given in \eqref{eqdq} appears when a first integral $H(x,y)$ of system \eqref{eq3} with $\lambda=\lambda^*$ of the form $H(x,y)=\sqrt{x^2+y^2}+o\left( \sqrt{x^2+y^2} \right)$ is used. Some of the $v_{2j+1}(\mathbf{\lambda})$ might be identically null
and in such a case we take, by default, the corresponding
$b_{2j+1}(h,\la)$ constant equal to $1$. We note that the
Poincar\'e--Liapunov constants can be computed by algebraic methods,
but the computations are usually cumbersome. Other useful remarks
are that, since for $\la=\la^*$ system (\ref{eq3}) has a center at
the origin, we must have $d(h;\la^*) \equiv 0$ and, consequently,
$$v_{2j+1}(\mathbf{\la}^*)=0\quad  \mbox{for}\quad j=\overline{0, N},$$ and  that the functions $\ h^{2j+1}b_{2j+1}(h,\mathbf{\lambda})\ $
for $j=\overline{0, N}$ are linearly independent on a sufficiently
small neighborhood of $(h,{\la})=(0,{\la}^*)$.\\

We consider now a small bifurcation parameter $\varepsilon$ and that, in \eqref{eq3}, $\la=\la(\varepsilon)$ depends analytically on $\varepsilon$ such that $\la(0)=\la^*$. We denote by $\la_{i,0}$ the $i^{th}$ coordinate of the point
$\la^{*} \in \mathbb{R}^n$, that is, $\mathbf{\la}^*=(\la_{1,0},\la_{2,0}, \ldots ,\la_{n,0})$. In
addition, we take the series expansions
\[ \lambda_i(\varepsilon) \, = \, \sum_{\ell \geq 0} \lambda_{i,\ell} \, \varepsilon^\ell, \]
for some reals $\lambda_{i,\ell}$. So now we see system \eqref{eq3},
i.e.
\begin{equation} \dot{x} \, = \, -y + \la_1(\varepsilon) x + P(x,y, \mathbf{\lambda}(\varepsilon)), \quad
\dot{y} \, = \, x + \la_1(\varepsilon) y + Q(x,y,\mathbf{\lambda}(\varepsilon)), \label{eq3-e}
\end{equation}
as a one-parameter analytic perturbation of the
period annulus surrounding the origin of system (\ref{eq3}) when
$\mathbf{\la} \, = \, \mathbf{\la}^*$. We emphasize that
\eqref{eq3-e} depends on the parameters $\tilde
\la=(\la_{i,\ell}\, :\, i=\overline{1,n},\ \ell\geq 0)$. Hence
\eqref{eq3-e} is like system \eqref{eq2} from the Introduction.

The displacement map of \eqref{eq3-e} is $d(h;\mathbf{\la}(\eps))$
and we assume that its Taylor series expansion in a neighborhood of
$\eps=0$ takes the form
\begin{equation} \label{eqdisp}
d(h;\mathbf{\la}(\varepsilon)) \, =  \, M_k(h) \, \varepsilon^k \, +
\, \mathcal{O}(\varepsilon^{k+1}),
\end{equation}
where $M_k(h)$ is the Melnikov function (of order $k\geq 1$). As we
have  remarked in the Introduction, $M_k(h)$ is in fact defined and
analytic not only in a small neighborhood of $h=0$, but  on the
whole $\Sigma$, i.e. on the interval $[0,h_1)$. The notations
$\Sigma$ and $h_1$ are given in the Introduction. This analyticity
property is a consequence of the Global Bifurcation Lemma, referred
as Lemma 2.2 in the work \cite{ChiJac}.

In order to present the main result of this Section, we need to
identify the coefficients of the power series expansions of the
Poincar\'{e}--Liapunov constants:
\begin{equation} \label{eqve} v_{2j+1}(\mathbf{\la}(\eps)) \, = \, \sum_{r\geq 1} \ov_{2j+1,r} \, \eps^r,\quad j=\overline{0,N}. \end{equation}
It can be shown that, for each $j=\overline{0,N}$,  $\ov_{2j+1,r}$
are polynomials in $(\la_{i,\ell}\,:\ i=\overline{1,n},\
\ell=\overline{0,r})$.\\

\begin{theorem} \label{thcj}
There are $N+1$ linearly independent functions $h^{2j+1}B_{2j+1}(h)$ which are analytic in $[0,h_1)$ %for $h$ in the whole period annulus
and with $B_{2j+1}(0)$ a nonzero constant for $j=\overline{0,N}$,
such that the Melnikov function of system \eqref{eq3-e} writes as
\begin{equation} \label{eqMkj} M_k(h) \, = \, \sum_{j=0}^{N} \ov_{2j+1,k} \, h^{2j+1}B_{2j+1}(h), \end{equation} where $M_k(h)$ is defined in \eqref{eqdisp}.
\end{theorem}
\begin{proof} This proof is inspired by the one given by Chicone and Jacobs in \cite{ChiJac} for the case that system (\ref{eq3}) is a quadratic system
written in Bautin normal form. We consider the functions
$b_{2j+1}(h,\la)$, for $j=\overline{0,N}$, defined in (\ref{eqdq})
which are analytic in a neighborhood of $h=0$. We define \[ B_{2j+1}(h) \, := \, b_{2j+1}(h,\la^*)\] which is an
analytic function in a neighborhood of $h=0$ and verifies that
$B_{2j+1}(0) \, = \, b_{2j+1}(0,\la^*)$ is a nonzero constant.
Hence, we have that the $N+1$ functions $h^{2j+1}B_{2j+1}(h)$ for
$j=\overline{0,N}$ are linearly independent because each of them has
a different subdegree in $h$. We have that
\[ b_{2j+1}(h,\la(\eps)) \, = \, B_{2j+1}(h) \, + \, \eps \, R_{2j+1}(h,\eps), \]
where $R_{2j+1}(h, \eps)$ is an analytic function in a neighborhood
of $(h,\eps)=(0,0)$. We substitute the latter expression of
$b_{2j+1}(h,\la(\eps))$ and the expansion of $v_{2j+1}(\la(\eps))$
given in (\ref{eqve}) in the expression (\ref{eqdq}) of the
displacement map, that is
\[ \begin{array}{l}
\displaystyle d(h;\mathbf{\lambda}(\eps)) \, = \, \displaystyle \sum_{j=0}^{N} v_{2j+1}(\mathbf{\lambda}(\eps))\,  h^{2j+1} \, b_{2j+1}(h,\mathbf{\lambda}(\eps)) \\ \quad = \, \displaystyle \sum_{j=0}^{N} \left[ \sum_{r\geq 1} \ov_{2j+1,r} \, \eps^r \right]  \, h^{2j+1} \, \Big( B_{2j+1}(h) \, + \, \eps \, R_{2j+1}(h,\eps) \Big)
%\\ \quad = \, \displaystyle \sum_{j=0}^{N} \sum_{r \geq 1}  \left( \ov_{2j+1,r} \, h^{2j+1}B_{2j+1}(h) \right) \eps^r \, + \, \ov_{2j+1,r}\,  h^{2j+1}R_{2j+1}(h,\eps) \, \eps^{r+1}
\\ \quad = \, \displaystyle  \sum_{r \geq 1} \sum_{j=0}^{N} \left( \ov_{2j+1,r} \, h^{2j+1}B_{2j+1}(h) \right) \eps^r \, + \, \ov_{2j+1,r}\,  h^{2j+1}R_{2j+1}(h,\eps) \, \eps^{r+1} .
\end{array} \]
By (\ref{eqdisp}) we are under the assumption that the lowest order
term in the expansion of $d(h;\mathbf{\lambda}(\eps))$ in powers of
$\eps$ corresponds to the power $\eps^k$. If $k=1$ we conclude that
\[ M_1(h) \, = \, \sum_{j=0}^{N}  \ov_{2j+1,1} \, h^{2j+1}B_{2j+1}(h) .\]
If $k>1$, we deduce that, for $r=\overline{1,k-1}$, we have
\[ \sum_{j=0}^{N} \ov_{2j+1,r}\,  h^{2j+1}B_{2j+1}(h) \equiv 0. \]  Since the functions $h^{2j+1}B_{2j+1}(h)$ for $j=\overline{0, N}$, are linearly independent,
 we deduce that $\ov_{2j+1,r}\, = \, 0$ for $r=\overline{1,k-1}$ and
$j=\overline{0, N}$. Therefore,
\[ d(h;\mathbf{\lambda}(\eps)) \, = \, \left( \sum_{j=0}^{N} \ov_{2j+1,k} \, h^{2j+1}B_{2j+1}(h) \right) \, \eps^k \, + \, \mathcal{O}(\eps^{k+1}). \] Equating the coefficients of $\eps^k$ in this expression and in (\ref{eqdisp}), we conclude that
\[ M_k(h) \, = \, \sum_{j=0}^{N} \ov_{2j+1,k} \, h^{2j+1}B_{2j+1}(h). \]
\end{proof}

\bigskip

\noindent We refer to the $N+1$ linearly independent functions
$h^{2j+1}B_{2j+1}(h)$, for $j=\overline{0,N}$, as the {\it Bautin
functions} associated to family  (\ref{eq3}). As a consequence of
Theorem \ref{thcj}, we have that if one knows the Bautin functions,
the study of the Melnikov function $M_k(h)$ reduces to the study of
the coefficients $\ov_{2j+1,k}$. We remark that since $v_{2j+1}(\la)$ are polynomials in $\la$, we have that $\ov_{2j+1,k}$ are polynomials in $(\la_{i,\ell}\, :\, i=\overline{1,n},\, \ell=\overline{1,k})$ and that there is a recursive way to give the expression of $\ov_{2j+1,k}$ if $k$ is high enough.
In order to make this statement more precise, we make some
notations. Let $\Lambda_k$ be the real algebraic manifold
%\[ \Lambda_k \, : = \, \left\{(\lambda_{i,\ell})_{\begin{array}{ll}i=\overline{1,n}\\ \ell=\overline{1,k}\end{array}} \ :\  M_r(h)\equiv 0,\  r<k
%\right\}. \]
\[ \Lambda_k \, : = \, \left\{(\lambda_{i,\ell}\, :\, i=\overline{1,n},\, \ell=\overline{1,k})  \quad / \quad M_r(h)\equiv
0,\, r=\overline{1,k-1} \right\}\subseteq \mathbb{R}^{n k} \]
and we consider the map $\phi_k:\, \Lambda_k \to \mathbb{R}^{N+1}$
given by
\begin{equation} \label{eqmap} \phi_k: \, \Lambda_k \mapsto \left(\ov_{2j+1,k}\, :\, {j=\overline{0,N}}\right).\end{equation}
It is important to study the range of the map $\phi_k$ for any
$k\geq 1$. If it is possible to choose $k^*$ to be the smallest $k$
such that the range of the map $\phi_{k^*}$ is equal or contains the
range of $\phi_k$ for any other $k$, then $M_{k^*}(h)$ will be the
{\it essential Melnikov function} and $k^*$ will be the {\it
essential order}. After choosing $k^*$, we choose the {\it essential
parameters}, which are a parametrization of a submanifold of
$\Lambda_{k^*}$ with the lowest possible dimension on which
$\phi_{k^*}$ attains the maximal range. In fact, we fix the values
of the {\it non-essential parameters} (most of them will be taken to
be $0$) in order that the expression of $\phi_{k^*}$ is the simplest
possible but still maintains the maximal range.

\section{Essential perturbations of quadratic and cubic systems \label{sect3}}

\subsection{Essential perturbations of quadratic centers}

In this paragraph we consider ${\rm d}=2$ and we will give our results for
quadratic systems like \eqref{eq2} in the standard Bautin form.
 In fact, after an appropriate affine transformation  (analytic with
respect to $\eps$), system \eqref{eq2} for ${\rm d}=2$ can be put into the
standard Bautin form
\begin{equation}\label{eq:qsB0}
\begin{array}{ll}
\dot{x}=\la_1x-y-\la_3x^2+\left(2\la_2+\la_5\right)xy+\la_6y^2,\\
\dot{y}=x+\la_1y+\la_2x^2+\left(2\la_3+\la_4\right)xy-\la_2y^2,
\end{array}
\end{equation}
where the coefficients $\la(\eps)=(\la_1(\eps), ... , \la_6(\eps))$
are analytic functions for $|\eps|$ sufficiently small and such
that, for $\eps=0$ (i.e. for $\la(0)=(\la_{1,0}, ... , \la_{6,0})$),
system \eqref{eq:qsB0} has a center at the origin. We will apply
Theorem \ref{thcj} and other ideas presented in Section \ref{sect2}. As it is proved in \cite{Bautin}, see also \cite{Sch} and references therein, there are
four Poincar\'{e}--Liapunov constants of \eqref{eq:qsB0} (hence $N=3$
in this case) and they have the expressions
\begin{eqnarray*}
{v}_1(\la)&=&\la_1,\\
{v}_3(\la)&=&\la_5(\la_3-\la_6),\\
{v}_5(\la)&=&\la_2\la_4(\la_3-\la_6)(\la_4+5\la_3-5\la_6),\\
{v}_7(\la)&=&\la_2\la_4(\la_3-\la_6)^2(\la_3\la_6-2\la_6^2-\la_2^2).
\end{eqnarray*}
For further use we write here again equation \eqref{eqve} for our
four Poincar\'{e}--Liapunov constants
\begin{equation*} \label{eqve2} v_{2j+1}(\mathbf{\la}(\eps)) \, = \, \sum_{r\geq 1} \ov_{2j+1,r} \, \eps^r,\quad j=\overline{0,3}, \end{equation*}
 and
$$\la_i(\eps)=\sum_{j \geq 0} \la_{i,j}\eps^j.$$

Denote by $h B_1(h)$, $h^3 B_3(h)$, $h^5 B_5(h)$, $h^7 B_7(h)$ the four
Bautin functions for the family of quadratic systems in the standard
Bautin form. Applying Theorem \ref{thcj} we have that the first
Melnikov function has the expression
$$M_1(h)=\ov_{1,1} h B_1(h) + \ov_{3,1} h^3 B_3(h) + \ov_{5,1} h^5 B_5(h) + \ov_{7,1} h^7 B_7(h).$$
Moreover, if $M_j(h)\equiv0$ for $j<k$, then
$$M_k(h)=\ov_{1,k} h B_1(h) + \ov_{3,k} h^3 B_3(h) + \ov_{5,k} h^5 B_5(h) + \ov_{7,k} h^7 B_7(h).$$
The reals $\ov_{1,k}, \ \ov_{3,k}, \ \ov_{5,k},\ \ov_{7,k}$ will be
called here the coefficients of the Melnikov function $M_k$.\\

We remind that \cite{Bautin}, for $\eps=0$ (i.e. for $\la(0)=(\la_{1,0}, ... ,
\la_{6,0})$), system \eqref{eq:qsB0} has a center at the origin if
and only if one of the following relations holds  (first we indicate
the name used in literature for the corresponding center condition)

 (a) Lotka--Volterra: $\la_{3,0}=\la_{6,0}$;

(b) Symmetric (or Reversible): $\la_{2,0}=\la_{5,0}=0$;

(c) Hamiltonian: $\la_{4,0}=\la_{5,0}=0$;

(d) Darboux (or Codimension 4):
$\la_{5,0}=\la_{4,0}+5\la_{3,0}-5\la_{6,0}=\la_{3,0}\la_{6,0}-2\la_{6,0}^2-\la_{2,0}^2=0$.\\

\noindent In the next lemma we give the expressions of the
coefficients of the Melnikov functions, the essential order and the
essential parameters for all possible positions of a point
$(\la_{1,0}, ... , \la_{6,0})$ in the center variety. This lemma is
followed by a theorem which gives the essential perturbation and the
essential Melnikov function in each situation.

\begin{lemma}  \label{lem:CJlike}
For any integer $k\geq 1$, the following statements hold.
\begin{itemize}
\item[{\rm (i)}] Generic Lotka--Volterra: $\la_{1,0}=\la_{3,0}-\la_{6,0}=0$ and $\la_{5,0}\neq 0$.

If $\ov_{1,j}=\ov_{3,j}=0$ for $j=\overline{0,\ k-1}$, then
\begin{eqnarray*}
\ov_{1,k} & = & \la_{1,k}, \\
\ov_{3,k}&=&\la_{5,0}(\la_{3,k}-\la_{6,k}),\\
\ov_{5,k}&=&\la_{2,0}\la_{4,0}^2(\la_{3,k}-\la_{6,k}),\\
\ov_{7,k}&=&0\ .
\end{eqnarray*}
The essential order is $k^*=1$ and the essential parameters can be chosen to be $\la_{1,1}$ and $\la_{6,1}$.

\item[{\rm (ii)}] Generic symmetric: $\la_{1,0}=\la_{2,0}=\la_{5,0}=0$, $\la_{4,0}(\la_{3,0}-\la_{6,0})\neq 0$ and
$(\la_{4,0}+5\la_{3,0}-5\la_{6,0})^2+(\la_{3,0}\la_{6,0}-2\la_{6,0}^2)^2\neq 0$.

 If $\ov_{1,j}=\ov_{3,j}=\ov_{5,j}=\ov_{7,j}=0$ for $j=\overline{0,\ k-1}$, then
\begin{eqnarray*}
\ov_{1,k} & = & \la_{1,k}, \\
\ov_{3,k}&=&(\la_{3,0}-\la_{6,0})\la_{5,k},\\
\ov_{5,k}&=&\la_{4,0}^2(\la_{3,0}-\la_{6,0})(\la_{4,0}+5\la_{3,0}-5\la_{6,0})\la_{2,k},\\
\ov_{7,k}&=&\la_{4,0}(\la_{3,0}-\la_{6,0})^2(\la_{3,0}\la_{6,0}-2\la_{6,0}^2)\la_{2,k}.
\end{eqnarray*}
The essential order is $k^*=1$ and the essential parameters can be chosen to be $\la_{1,1}$, $\la_{2,1}$ and $\la_{5,1}$.

\item[{\rm (iii)}] Generic Hamiltonian: $\la_{1,0}=\la_{4,0}=\la_{5,0}=0$ and $\la_{2,0}(\la_{3,0}-\la_{6,0})\neq 0$.

If $\ov_{1,j}=\ov_{3,j}=\ov_{5,j}=0$ for $j=\overline{0,\ k-1}$,
then
\begin{eqnarray*}
\ov_{1,k} & = & \la_{1,k}, \\
\ov_{3,k}&=&(\la_{3,0}-\la_{6,0})\la_{5,k},\\
\ov_{5,k}&=&\la_{2,0}(\la_{3,0}-\la_{6,0})^2\la_{4,k},\\
\ov_{7,k}&=&\la_{2,0}(\la_{3,0}-\la_{6,0})^2(\la_{3,0}\la_{6,0}-2\la_{6,0}^2-\la_{2,0}^2)\la_{4,k}.
\end{eqnarray*}
The essential order is $k^*=1$ and the essential parameters can be chosen to be $\la_{1,1}$, $\la_{4,1}$ and $\la_{5,1}$.

\item[{\rm (iv)}] Generic Darboux: $\la_{1,0}=\la_{5,0}=\la_{4,0}+5\la_{3,0}-5\la_{6,0}=\la_{3,0}\la_{6,0}-2\la_{6,0}^2-\la_{2,0}^2=0$ and $\la_{2,0}\la_{4,0}(\la_{3,0}-\la_{6,0})\neq 0$. Then
\begin{eqnarray*}
\ov_{1,1} & = & \la_{1,1}, \\
\ov_{3,1}&=&(\la_{3,0}-\la_{6,0})\la_{5,1},\\
\ov_{5,1}&=&\la_{2,0}\la_{4,0}(\la_{3,0}-\la_{6,0})(\la_{4,1}+5\la_{3,1}-5\la_{6,1}),\\
\ov_{7,1}&=&\la_{2,0}\la_{4,0}(\la_{3,0}-\la_{6,0})^2(\la_{3,0}\la_{6,1}+\la_{6,0}\la_{3,1}-4\la_{6,0}\la_{6,1}-2\la_{2,0}\la_{2,1}).
\end{eqnarray*}
The essential order is $k^*=1$ and the essential parameters can be chosen to be $\la_{1,1}$, $\la_{2,1}$, $\la_{4,1}$ and $\la_{5,1}$.

\item[{\rm (v)}] Symmetric Lotka--Volterra: $\la_{1,0}=\la_{2,0}=\la_{5,0}=\la_{3,0}-\la_{6,0}=0$ and $\la_{4,0}\neq 0$.

Then $\ov_{1,1} \, = \, \la_{1,1}$ and $\ov_{3,1}=\ov_{5,1}=\ov_{7,1}=0$. If $\la_{1,1}=0$, then
\begin{eqnarray*}
\ov_{1,2}& = & \la_{1,2}, \\
\ov_{3,2}&=&(\la_{3,1}-\la_{6,1})\la_{5,1},\\
\ov_{5,2}&=&\la_{4,0}^2 \la_{2,1}(\la_{3,1}-\la_{6,1}),\\
\ov_{7,2}&=&0.
\end{eqnarray*}

If $\ov_{5,j}=0$ for $j=\overline{0,\ k-1}$ with $k \geq 2$, then
$$\ov_{7,k}=0.$$
The essential order is $k^*=2$, and the essential parameters can be chosen to be $\la_{1,2}$, $\la_{2,1}$ and $\la_{5,1}$, taking $\la_{3,1}=1$.

\item[{\rm (vi)}] Symmetric Hamiltonian: $\la_{1,0}=\la_{2,0}=\la_{4,0}=\la_{5,0}=0$ and $\la_{3,0}-\la_{6,0}\neq 0$.

Then $\ov_{1,1}=\la_{1,1}$, $\ov_{3,1}=(\la_{3,0}-\la_{6,0})\la_{5,1}, \quad \ov_{5,1}=\ov_{7,1}=0$.

If $\ov_{1,1}=\ov_{3,1}=0$ then
\begin{eqnarray*}
\ov_{1,2} & = & \la_{1,2}, \\
\ov_{3,2}&=&(\la_{3,0}-\la_{6,0})\la_{5,2}, \\
\ov_{5,2}&=& (\la_{3,0}-\la_{6,0})^2\la_{2,1}\la_{4,1},\\
\ov_{7,2}&=&
(\la_{3,0}-\la_{6,0})^2(\la_{3,0}\la_{6,0}-2\la_{6,0}^2)\la_{2,1}\la_{4,1}.\end{eqnarray*}

Moreover, if $\ov_{5,j}=0$ for $j=\overline{1 ,\ k-1}$, with $k \geq
2$ then there exists some $i\in \{1,2,...,k-1\}$ such that
\begin{eqnarray*}
\ov_{5,k}&=& (\la_{3,0}-\la_{6,0})^2\la_{2,i}\la_{4,k-i},\\
\ov_{7,k}&=&
(\la_{3,0}-\la_{6,0})^2(\la_{3,0}\la_{6,0}-2\la_{6,0}^2)\la_{2,i}\la_{4,k-i}.\end{eqnarray*}
The essential order is $k^*=2$ and the essential parameters can be chosen to be $\la_{1,2}$, $\la_{4,1}$ and $\la_{5,2}$, taking $\la_{2,1}=1$.
\item[{\rm (vii)}] Symmetric Darboux: $\la_{1,0}=\la_{2,0}=\la_{5,0}=\la_{4,0}+5\la_{3,0}-5\la_{6,0}=\la_{3,0}\la_{6,0}-2\la_{6,0}^2=0$ and $\la_{4,0}\neq 0$.

Then  $\la_{3,0}-\la_{6,0}\neq 0$, $\la_{3,0}\neq 0$, and
$$ \ov_{1,1}=\la_{1,1}, \quad \ov_{3,1}=(\la_{3,0}-\la_{6,0})\la_{5,1}, \quad
\ov_{5,1}=\ov_{7,1}=0.$$ If $\ov_{1,1}=\ov_{3,1}=0$ then
\begin{eqnarray*}
\ov_{1,2} & = & \la_{1,2}, \\
\ov_{3,2}&=&(\la_{3,0}-\la_{6,0})\la_{5,2}, \\
\ov_{5,2}&=& \la_{4,0}(\la_{3,0}-\la_{6,0})\la_{2,1}\left(\la_{4,1}+5\la_{3,1}-5\la_{6,1}\right),\\
\ov_{7,2}&=& \la_{4,0}(\la_{3,0}-\la_{6,0})^2\la_{2,1}(\la_{3,0}\la_{6,1}+\la_{6,0}\la_{3,1}).
\end{eqnarray*}
The essential order is $k^*=2$ and the essential parameters can be chosen to be $\la_{1,2}$,  $\la_{5,2}$, $\la_{4,1}$ and $\la_{6,1}$ taking $\la_{2,1}=1$.

\item[{\rm (viii)}] Hamiltonian Lotka--Volterra: $\la_{1,0}=\la_{4,0}=\la_{5,0}=\la_{3,0}-\la_{6,0}=0$ and $\la_{2,0}\neq 0$.
Then $\ov_{1,1}=\la_{1,1}$,and $\ov_{3,1}=\ov_{5,1}=\ov_{7,1}=0$.
If $\ov_{1,1}=0$ then
\begin{eqnarray*} \ov_{1,2}=\la_{1,2}, \quad
\ov_{3,2}&=&(\la_{3,1}-\la_{6,1})\la_{5,1}, \quad \ov_{5,2}=\ov_{7,2}=0.
\end{eqnarray*}

If $\la_{3,1}-\la_{6,1}\neq 0$ and $\la_{1,2}=\la_{5,1}=0$, then $\ov_{1,2}=\ov_{3,2}=0$ and
\begin{eqnarray*}
\ov_{1,3}&=&\la_{1,3},\\
\ov_{3,3}&=&(\la_{3,1}-\la_{6,1})\la_{5,2}, \\
\ov_{5,3}&=& \la_{2,0}(\la_{3,1}-\la_{6,1})\left(\la_{4,1}+5\la_{3,1}-5\la_{6,1}\right)\la_{4,1},\\
\ov_{7,3}&=& \la_{2,0}\left(\la_{6,0}^2+\la_{2,0}^2\right)(\la_{3,1}-\la_{6,1})^2\la_{4,1}.
\end{eqnarray*}
The essential order is $k^*=3$ and the essential parameters can be chosen to be $\la_{1,3}$, $\la_{3,1}$, $\la_{5,2}$ and $\la_{4,1}$.

Moreover if $\ov_{1,3}=\ov_{3,3}=\ov_{5,3}=\ov_{7,3}=0$ then
\begin{eqnarray*}
\ov_{1,4}&=&\la_{1,4},\\
\ov_{3,4}&=&(\la_{3,1}-\la_{6,1})\la_{5,3}, \\
\ov_{5,4}&=& \la_{2,0}(\la_{3,1}-\la_{6,1})(5\la_{3,1}-5\la_{6,1})\la_{4,2},\\
\ov_{7,4}&=&\la_{2,0}\left(\la_{6,0}^2+\la_{2,0}^2\right)(\la_{3,1}-\la_{6,1})^2\la_{4,2}.
\end{eqnarray*}
and if $\la_{1,2}=\la_{1,3}=\la_{1,4}=\la_{3,1}-\la_{6,1}=
\la_{5,1}=\la_{3,2}-\la_{6,2}=0$ then
$\ov_{1,j}=\ov_{3,j}=\ov_{5,j}=\ov_{7,j}=0$ for $j=\overline{1,4}$
and
\begin{eqnarray*}
\ov_{1,5}&=&\la_{1,5},\\
\ov_{3,5}&=&(\la_{3,3}-\la_{6,3})\la_{5,2}, \\
\ov_{5,5}&=& \la_{2,0}\la_{4,1}^2(\la_{3,3}-\la_{6,3}),\\
\ov_{7,5}&=&0.
\end{eqnarray*}

\item[{\rm (ix)}] Symmetric Hamiltonian Lotka--Volterra (Hamiltonian triangle): $\la_{1,0}=\la_{2,0}=\la_{4,0}=\la_{5,0}=\la_{3,0}-\la_{6,0}=0$ and $\la_{6,0}\neq 0$.
Then
\begin{eqnarray*}  \ov_{3,1}&=&\ov_{5,1}=\ov_{7,1}=0,\\
\ov_{3,2}&=&(\la_{3,1}-\la_{6,1})\la_{5,1}, \quad
\ov_{5,2}=\ov_{7,2}=0,\\
\ov_{5,3}&=&\ov_{7,3}=0.
\end{eqnarray*}
If $\la_{3,1}-\la_{6,1} \ne 0$ and $\ov_{1,j}=\ov_{3,j}=0$ for
$j=\overline{1,4}$ then
\begin{eqnarray*}
\ov_{1,4}&=&\la_{1,4},\\
\ov_{3,4}&=&(\la_{3,1}-\la_{6,1})\la_{5,3},\\
\ov_{5,4}&=&\la_{2,1}\la_{4,1}(\la_{3,1}-\la_{6,1})\left(\la_{4,1}+5\la_{3,1}-5\la_{6,1}\right),\\
\ov_{7,4}&=&\la_{6,0}^2\la_{2,1}\la_{4,1}(\la_{3,1}-\la_{6,1})^2.
\end{eqnarray*}
The essential order is $k^*=4$ and the essential parameters can be
chosen to be $\la_{1,4}$, $\la_{5,3}$, $\la_{2,1}$ and $\la_{4,1}$
taking $\la_{3,1}=1$. If $\la_{3,1}-\la_{6,1}=\la_{5,1}=\la_{5,2}=0$
and $\ov_{1,j}=0$ for $j=\overline{1,4}$ then
$\ov_{3,j}=\ov_{5,j}=\ov_{7,j}=0$ for $j=\overline{1,4}$ and
\begin{eqnarray*}
\ov_{1,5}&=&\la_{1,5},\\
\ov_{3,5}&=&(\la_{3,2}-\la_{6,2})\la_{5,3},\\
\ov_{5,5}&=&\la_{2,1}\la_{4,1}^2 (\la_{3,2}-\la_{6,2}),\\
\ov_{7,5}&=&0.
\end{eqnarray*}

\item[{\rm (x)}] Linear center: $\la_{1,0}=\la_{2,0}=\la_{3,0}=\la_{4,0}=\la_{5,0}=\la_{6,0}=0$. Then
$\ov_{1,1}= \la_{1,1}$ and $\ov_{3,1}=\ov_{5,1}=\ov_{7,1}=0$.
If $\ov_{1,1}=0$, then $\ov_{1,2}= \la_{1,2}$, $\ov_{3,2}= \la_{5,1}(\la_{3,1}-\la_{6,1})$  and $\ov_{5,2}=\ov_{7,2}=0$.

\medskip

\noindent If $\la_{3,1}-\la_{6,1} \ne 0$ and $\la_{1,2}=\la_{5,1}=0$, then $\ov_{1,2}=\ov_{3,2}=0$  and
$\ov_{1,3}= \la_{1,3}$, $\ov_{3,3}= \la_{5,2}(\la_{3,1}-\la_{6,1})$  and $\ov_{5,3}=\ov_{7,3}=0$.

\medskip

\noindent If $\la_{3,1}-\la_{6,1} \ne 0$ and $\la_{1,3}=\la_{5,2}=0$, then $\ov_{1,3}=\ov_{3,3}=0$  and
$\ov_{1,4}= \la_{1,4}$, $\ov_{3,4}= \la_{5,3}(\la_{3,1}-\la_{6,1})$,
$\ov_{5,4}=\la_{2,1}\la_{4,1}(\la_{3,1}-\la_{6,1})(\la_{4,1}+5(\la_{3,1}-\la_{6,1}))$ and $\ov_{7,4}=0$.

\medskip

\noindent If $(\la_{3,1}-\la_{6,1})\la_{2,1}\la_{4,1} \ne 0$, $\la_{1,4}=\la_{5,3}=0$, $\la_{4,1}=5(\la_{6,1}-\la_{3,1})$ then  $\ov_{1,4}=\ov_{3,4}=\ov_{5,4}=0$ and
$\ov_{1,5}= \la_{1,5}$, $\ov_{3,5}= \la_{5,4}(\la_{3,1}-\la_{6,1})$,
$\ov_{5,5}=\la_{2,1}\la_{4,1}(\la_{3,1}-\la_{6,1})(\la_{4,2}+5(\la_{3,2}-\la_{6,2}))$ and $\ov_{7,5}=0$.

\medskip

\noindent If $(\la_{3,1}-\la_{6,1})\la_{2,1}\la_{4,1} \ne 0$, $\la_{1,5}=\la_{5,4}=0$, $\la_{4,2}=5(\la_{6,2}-\la_{3,2})$ then $\ov_{1,5}=\ov_{3,5}=\ov_{5,5}=0$ and
\begin{eqnarray*}
\ov_{1,6}&=&\la_{1,6},\\
\ov_{3,6}&=&(\la_{3,1}-\la_{6,1})\la_{5,5},\\
\ov_{5,6}&=&\la_{2,1}(\la_{3,1}-\la_{6,1})^2 (\la_{4,3}+5(\la_{3,3}-\la_{6,3})),\\
\ov_{7,5}&=&\la_{2,1}(\la_{3,1}-\la_{6,1})^3 (\la_{2,1}^2-\la_{3,1}\la_{6,1}+2\la_{6,1}^2).
\end{eqnarray*}
The essential order is $k^*=6$ and the essential parameters can be chosen to be $\la_{1,6}$, $\la_{5,5}$, $\la_{4,3}$ and $\la_{2,1}$ taking $\la_{3,1}=\la_{6,1}+1$, $\la_{3,3}=\la_{6,3}$ and $\la_{6,1}=1/4$.

\end{itemize}
\end{lemma}

\begin{proof}  The cases (i) and (ii)  were already proved in \cite{ChiJac}.\\

(iii) Note that $\ov_{3,j}=\ov_{5,j}=0$ for $j=\overline{0,\ k-1}$
means that $\la_{5,j}=\la_{4,j}=0$ for $j=\overline{0,\ k-1}$. Then
 \begin{eqnarray*}\overline v_3\left(\la(\eps)\right)&=&\left[(\la_{3,0}-\la_{6,0})+O(\eps)\right]\left[ \la_{5,k}\eps^k+O(\eps^{k+1})\right],\\
\overline v_5\left(\la(\eps)\right)&=&\left[\la_{2,0}(\la_{3,0}-\la_{6,0})(5\la_{3,0}-5\la_{6,0})+O(\eps)\right]\left[ \la_{4,k}\eps^k+O(\eps^{k+1})\right],\\
\overline v_7\left(\la(\eps)\right)&=&\left[\la_{2,0}(\la_{3,0}-\la_{6,0})^2
(\la_{3,0}\la_{6,0}-2\la_{6,0}^2-\la_{2,0}^2)+O(\eps)\right]\left[ \la_{4,k}\eps^k+O(\eps^{k+1})\right]
\end{eqnarray*}
and the expressions of $\ov_{3,k}$, $\ov_{5,k}$ and $\ov_{7,k}$ easily follow. It is not difficult to see that the range of the map \eqref{eqmap} is the same for each $k$, hence the essential order is $k^*=1$.\\

(iv) In this case the range of \eqref{eqmap} for $k=1$ is $\mathbb{R}^4$, the largest possible. Hence the essential order is $k^*=1$.\\

(v) In this case we have
\begin{eqnarray*}\overline v_3\left(\la(\eps)\right)&=&\left[\eps(\la_{3,1}-\la_{6,1})+O(\eps^2)\right]\left[ \eps \la_{5,1}+O(\eps^{2})\right], \\
\overline v_5\left(\la(\eps)\right)&=&\left[ \eps \la_{2,1}+O(\eps^{2})\right]\left[\eps(\la_{3,1}-\la_{6,1})+O(\eps^2)\right]\left[\la_{4,0}^2+O(\eps)\right], \\
\overline v_7\left(\la(\eps)\right)&=&\left[ \eps \la_{2,1}+O(\eps^{2})\right]\left[\eps(\la_{3,1}-\la_{6,1})+O(\eps^2)\right]^2\left[-\la_{4,0}\la_{3,0}^2+O(\eps)\right].
\end{eqnarray*}
The coefficient of $\eps$ in each of the above expressions is null, and it is easy to identify the coefficient of $\eps^2$.

In order to see that if $\ov_{5,j}=0$ for $j=0,...,k-1$ then $\ov_{7,k}=0$, we note that $\overline v_7\left(\la(\eps)\right)=\overline v_5\left(\la(\eps)\right)\left[\eps(\la_{3,1}-\la_{6,1})+O(\eps^2)\right](-\frac{\la_{3,0}^2}{\la_{4,0}}+O(\eps))$.
 We have that if  $\ov_{5,j}=0$ for $j=0,...,k-1$ then $\overline v_5\left(\la(\eps)\right)$ has order $k$ in $\eps=0$. Hence $\overline v_7\left(\la(\eps)\right)$  has at least order $k+1$ in $\eps=0$, meaning that $\ov_{7,k}=0$.

Since the image of \eqref{eqmap} for $k=2$ is $\mathcal{R}=\{
(a,b,c,0)\ :\ (a,b,c)\in\mathbb{R}^3 \}$, from what we showed above
we deduce that for  $k\neq 2$ the image of \eqref{eqmap} is either
equal or contained in $\mathcal{R}$.
 Taking all these into account we deduce that the essential order is $k^*=2$. \\

 (vi) In this case we have
\begin{eqnarray*}\overline v_3\left(\la(\eps)\right)&=&\left[(\la_{3,0}-\la_{6,0})+O(\eps)\right]\left[ \eps \la_{5,1}+\eps^2 \la_{5,2}+O(\eps^{2})\right], \\
\overline v_5\left(\la(\eps)\right)&=&\left[ \eps \la_{2,1}+O(\eps^{2})\right]\left[\eps\la_{4,1}+O(\eps^2)\right]\left[(\la_{3,0}-\la_{6,0})^2+O(\eps)\right], \\
\overline v_7\left(\la(\eps)\right)&=&\left[ \eps \la_{2,1}+O(\eps^{2})\right]\left[\eps\la_{4,1}+O(\eps^2)\right]
\\ & &
\left[(\la_{3,0}-\la_{6,0})^2(\la_{3,0}\la_{6,0}-2\la_{6,0}^2)+O(\eps)\right].
\end{eqnarray*}
Identifying the coefficients of $\eps$ and $\eps^2$ we obtain in the expressions of $\ov_{j,1}$ and $\ov_{j,2}$ for $j=1,3,5,7$ given in the statement.

Assume that $\ov_{5,j}=0$ for $j=\overline{1 ,\ k-1}$. Then, from the above expressions we deduce that $\la_2(\eps)\la_4(\eps)$ has order $k$ in $\eps=0$. Hence, if $i$ is such that $\la_2(\eps)=\eps^i\la_{2,i}+O(\eps^{i+1})$ then $\la_4(\eps)=\eps^{k-i}\la_{4,k-i}+O(\eps^{k-i+1})$ and the coefficient of $\eps^k$ in their product is $\la_{2,i}\la_{4,k-i}$. The expressions of $\ov_{5,k}$ and $\ov_{7,k}$ follow from the above considerations. \\

(vii) Taking $\la_{2,1}=1$ and $\la_{3,1}=0$ the range of the map \eqref{eqmap} for $k=2$ is $\mathbb{R}^4$, the largest possible. Hence indeed $k^*=2$ is the essential order and $\la_{1,2}, \la_{4,1}, \la_{5,2}, \la_{6,1}$ are the essential parameters.\\

(viii) In this case we have
\begin{eqnarray*}\ov_3(\la(\eps))&=&\left[(\la_{3,1}-\la_{6,1})\eps +O(\eps^2)\right]\left[ \la_{5,1}\eps+\la_{5,2}\eps^2+\la_{5,3}\eps^3+O(\eps^4)\right],\\
\ov_5\left(\la(\eps)\right)&=&\left[\la_{2,0}(\la_{3,1}-\la_{6,1}) (\la_{4,1}+5\la_{3,1}-5\la_{6,1})\eps^2 +O(\eps^3)\right]\\
&& \left[ \la_{4,1}\eps+\la_{4,2}\eps^2+O(\eps^3)\right],\\
\ov_7\left(\la(\eps)\right)&=&\left[\la_{2,0}(\la_{3,1}-\la_{6,1})^2 (-\la_{6,0}^2-\la_{2,0}^2)\eps^2 +O(\eps^3)\right]\\
&& \left[ \la_{4,1}\eps+\la_{4,2}\eps^2+O(\eps^3)\right].
\end{eqnarray*}
Identifying the coefficients of $\eps$, $\eps^2$, $\eps^3$, $\eps^4$
and $\eps^5$ we obtain the expressions of $\ov_{i,j}$ given in
the statement. If $\la_{3,1}-\la_{6,1}\neq 0$ and
$\la_{1,2}=\la_{5,1}=0$ then the closure of the range of the map
(\ref{eqmap}) for $k^*=3$ is $\mathbb{R}^4$. In fact the range is
$\mathbb{R}^4 \setminus \mathcal{R}_1 \cup  \mathcal{R}_2$ where
$\mathcal{R}_1= \{ (a,b,c,d) \in \mathbb{R}^4 \ | \ d \ne 0, \
(\la_{6,0}^2 +\la _{2,0}^2) c -5 d=0 \}$ and  $\mathcal{R}_2= \{
(a,b,c,d) \in \mathbb{R}^4 \ | \ c \ne 0, \ d=0 \}$. We continue our
study giving other cases that recover the gaps of the previous
range. If $\la_{3,1}-\la_{6,1}\neq 0$ and
$\ov_{1,j}=\ov_{3,j}=\ov_{5,j}=\ov_{7,j}=0$ for $j=1,2,3$ then the
range of the map (\ref{eqmap}) for $k^*=4$ is
$\mathcal{R}_1$.
If $\la_{3,1}-\la_{6,1}= \la_{5,1}=\la_{3,2}-\la_{6,2}=0$ and $\ov_{1,j}=\ov_{3,j}=\ov_{5,j}=\ov_{7,j}=0$
for $j=\overline{1,4}$ then the range of the map (\ref{eqmap}) for $k^*=5$ is $\mathcal{R}_2$.\\

(ix) By the same reasonings of the preceding cases we obtain the
expressions of $\ov_{i,j}$ given in the statement. If
$\la_{3,1}-\la_{6,1} \ne 0$ and $\ov_{1,j}=\ov_{3,j}=0$ for
$j=\overline{1,4}$ then the closure of the range of the map
(\ref{eqmap}) for $k^*=3$ is $\mathbb{R}^4$. In fact the range is
$\mathbb{R}^4 \setminus \mathcal{R}_1$ where $\mathcal{R}_1= \{
(a,b,c,d) \in \mathbb{R}^4 \ | \ c \ne 0, \ d=0 \}$. We continue our
study giving the case that recover the gap of the previous range. If
$\la_{3,1}-\la_{6,1}= \la_{5,1}=\la_{5,2}=0$ and $\ov_{1,j}=0$ for
$j=\overline{1,4}$ then $\ov_{3,j}=\ov_{5,j}=\ov_{7,j}=0$ for
$j=\overline{1,4}$ and the range of the map (\ref{eqmap}) for
$k^*=5$ is
$\mathcal{R}_1$.\\

(x) In this case we have
\begin{eqnarray*}\ov_3(\la(\eps))&=&\left[(\la_{3,1}-\la_{6,1})\eps +O(\eps^2)\right]\\  &&
\left[ \la_{5,1}\eps+\la_{5,2}\eps^2+\la_{5,3}\eps^3+\la_{5,4}\eps^4+\la_{5,5}\eps^5+O(\eps^6)\right],\\
\ov_5\left(\la(\eps)\right)&=&\left[\la_{2,1}\eps +O(\eps^2)\right]\left[\la_{4,1}\eps +O(\eps^2)\right]
\left[(\la_{3,1}-\la_{6,1})\eps +O(\eps^2)\right]\\ &&
\left[(\la_{4,1}+5(\la_{3,1}-\la_{6,1}))\eps+(\la_{4,2}+5(\la_{3,2}-\la_{6,2}))\eps^2 \right. \\ &&
+ \left. (\la_{4,3}+5(\la_{3,3}-\la_{6,3}))\eps^3 +O(\eps^4)\right],\\
\ov_7\left(\la(\eps)\right)&=&\left[\la_{2,1}\eps +O(\eps^2)\right]\left[\la_{4,1}\eps +O(\eps^2)\right]
\left[(\la_{3,1}-\la_{6,1})\eps +O(\eps^2)\right]\\
&& \left[ (\la_{3,1}\la_{6,1}- 2 \la_{6,1}^2 - \la_{2,1}^2) \eps^2 +O(\eps^3)\right].
\end{eqnarray*}
Identifying the coefficients of $\eps$, $\eps^2$, $\eps^3$,
$\eps^4$, $\eps^5$ and $\eps^6$ we obtain the expressions of
$\ov_{i,j}$ given in the statement. If
$(\la_{3,1}-\la_{6,1})\la_{2,1}\la_{4,1} \ne 0$ and $\la_{1,j}=0$
for $j=\overline{1,5}$, $\la_{5,j}=0$ for $j=\overline{1,4}$,
$\la_{4,j}=5 (\la_{6,j}-\la_{3,j})$ for $j=1,2$, the range of the
map (\ref{eqmap}) for $k^*=6$ is $\R^4$, the largest possible. Hence
indeed $k^*=6$ is the essential order and $\la_{1,6}$, $\la_{5,5}$,
$\la_{4,3}$ and $\la_{2,1}$ are the essential parameters.
\end{proof}

\begin{theorem}\label{th:ep2}
The essential perturbations and the essential Melnikov function
are:
\begin{itemize}
\item[{\rm (i)}] Generic Lotka--Volterra center: $\la_{1,0}=\la_{3,0}-\la_{6,0}=0$ and $\la_{5,0}\neq 0$

\begin{equation*}\label{eq:gLV}
\begin{array}{ll}
\dot{x}=-y-\la_{6,0}x^2+\left(2\la_{2,0}+\la_{5,0}\right)xy+\la_{6,0}y^2+\eps\left(  \la_{1,1}x+\la_{6,1}y^2\right),\\
\dot{y}=x+\la_{2,0}x^2+\left(2\la_{6,0}+\la_{4,0}\right)xy-\la_{2,0}y^2
+ \eps \la_{1,1}y.
\end{array}
\end{equation*}
The corresponding essential Melnikov function is the first one and it has the form $$M_1(h)=\la_{1,1} h B_1(h)+\la_{6,1} h^3 \tilde B_3(h).$$

\item[{\rm (ii)}] Generic symmetric center: $\la_{1,0}=\la_{2,0}=\la_{5,0}=0$, $\la_{4,0}(\la_{3,0}-\la_{6,0})\neq 0$ and
$(\la_{4,0}+5\la_{3,0}-5\la_{6,0})^2+(\la_{3,0}\la_{6,0}-2\la_{6,0}^2)^2\neq 0$.

\begin{equation*}\label{eq:gSy}
\begin{array}{ll}
\dot{x}=-y-\la_{3,0}x^2+\la_{6,0}y^2+\eps\left(  \la_{1,1}x+\left(2\la_{2,1}+\la_{5,1}\right)xy\right),\\
\dot{y}=x+\left(2\la_{3,0}+\la_{4,0}\right)xy+ \eps \left(
\la_{1,1}y+\la_{2,1}x^2 -\la_{2,1}y^2\right).
\end{array}
\end{equation*}
The corresponding essential Melnikov function is the first one and it has the form, when $\la_{4,0} + 5 \la_{3,0} -5 \la_{6,0} \neq 0$,
$$M_1(h)=\la_{1,1}h B_1(h)+\la_{5,1} h^3 B_3(h) + \la_{2,1}  h^5 \tilde B_5(h), $$ and, when $\la_{4,0} + 5 \la_{3,0} -5 \la_{6,0} = 0$,
$$M_1(h)=\la_{1,1}h B_1(h)+\la_{5,1} h^3 B_3(h) + \la_{2,1}  h^7 B_7(h).$$

\item[{\rm (iii)}] Generic Hamiltonian center: $\la_{1,0}=\la_{4,0}=\la_{5,0}=0$ and $\la_{2,0}(\la_{3,0}-\la_{6,0})\neq
0$

\begin{equation*}\label{eq:gH}
\begin{array}{ll}
\dot{x}=-y-\la_{3,0}x^2+2\la_{2,0}xy+\la_{6,0}y^2+\eps\left(   \la_{1,1}x +\la_{5,1}xy\right),\\
\dot{y}=x+\la_{2,0}x^2+2\la_{3,0}xy-\la_{2,0}y^2 +\eps\left( \la_{1,1}y
+\la_{4,1}xy\right).
\end{array}
\end{equation*}
The corresponding essential Melnikov function is the first one and it has the form
$$M_1(h)=\la_{1,1} h B_1(h) + \la_{5,1} h^3 B_3(h) + \la_{4,1}  h^5 \tilde B_5(h).$$

\item[{\rm (iv)}] Generic Darboux center: $\la_{1,0}=\la_{5,0}=\la_{4,0}+5\la_{3,0}-5\la_{6,0}=\la_{3,0}\la_{6,0}-2\la_{6,0}^2-\la_{2,0}^2=0$ and
$\la_{2,0}\la_{4,0}(\la_{3,0}-\la_{6,0})\neq 0$

\begin{equation*}\label{eq:gD}
\begin{array}{ll}
\dot{x}=-y-\la_{3,0}x^2+2\la_{2,0}xy+\la_{6,0}y^2+\eps\left(  \la_{1,1}x+(2\la_{2,1}+\la_{5,1})xy\right),\\
\dot{y}=x+\la_{2,0}x^2+\left(7\la_{6,0}-5\la_{3,0}\right)xy-\la_{2,0}y^2
+ \eps \left(\la_{1,1}y+\la_{4,1}xy\right).
\end{array}
\end{equation*}
where $(\la_{3,0}\la_{6,0}-2\la_{6,0}^2-\la_{2,0}^2)=0$.
The corresponding essential Melnikov function is the first one and it has the form
$$M_1(h)=\la_{1,1}h B_1(h)+\la_{5,1} h^3 B_3(h) + \la_{4,1}  h^5 B_5(h)+ \la_{2,1} h^7 B_7(h) .$$

\item[{\rm (v)}] Symmetric Lotka--Volterra center: $\la_{1,0}=\la_{2,0}=\la_{5,0}=\la_{3,0}-\la_{6,0}=0$ and  $\la_{4,0}\neq 0$

\begin{equation*}\label{eq:SLV}
\begin{array}{ll}
\dot{x}=-y-\la_{6,0}x^2+\la_{6,0}y^2+\eps\left(  2\la_{2,1}+\la_{5,1}\right)xy+\eps^2\la_{1,2}x,\\
\dot{y}=x+\left(2\la_{6,0}+\la_{4,0}\right)xy + 2\eps xy+\eps \la_{2,1}
\left(x^2-y^2\right)+\eps^2\la_{1,2}y.
\end{array}
\end{equation*}
The corresponding essential Melnikov function is the second one and it has the form
$$M_2(h)=\la_{1,2} h B_1(h) + \la_{5,1} h^3 B_3(h) + \la_{2,1}  h^5 B_5(h) .$$

\item[{\rm (vi)}] Symmetric Hamiltonian center: $\la_{1,0}=\la_{2,0}=\la_{4,0}=\la_{5,0}=0$ and  $\la_{3,0}-\la_{6,0}\neq 0$

\begin{equation*}\label{eq:SH}
\begin{array}{ll}
\dot{x}=-y-\la_{3,0}x^2+\la_{6,0}y^2+\eps \ 2xy+\eps^2 \left( \la_{1,2}x +\la_{5,2}xy\right),\\
\dot{y}=x+2\la_{3,0}xy + \eps
\left(x^2+\la_{4,1}xy-y^2\right)+\eps^2\la_{1,2}y.
\end{array}
\end{equation*}
The corresponding essential Melnikov function is the second one and it has the form
$$M_2(h)=\la_{1,2}h B_1(h)+\la_{5,2} h^3 B_3(h) + \la_{4,1} h^5 \tilde B_5(h).$$

\item[{\rm (vii)}] Symmetric Darboux center: $\la_{1,0}=\la_{2,0}=\la_{5,0}=\la_{4,0}+5\la_{3,0}-5\la_{6,0}=\la_{3,0}\la_{6,0}-2\la_{6,0}^2=0$ and $\la_{4,0}\neq 0$

\begin{equation*}\label{eq:SD}
\begin{array}{ll}
\dot{x}=-y-\la_{3,0}x^2+\la_{6,0}y^2+\eps\left(  2xy+\la_{6,1}y^2\right)+\eps^2\left(\la_{1,2}x+\la_{5,2}xy\right),\\
\dot{y}=x+\left(5\la_{6,0}-3\la_{3,0}\right)xy + \eps
\left(x^2+\la_{4,1}xy-y^2\right)+\eps^2\la_{1,2}y.
\end{array}
\end{equation*}
The corresponding essential Melnikov function is the second one and it has the form
$$M_2(h)=\la_{1,2} h B_1(h)+\la_{5,2} h^3 B_3(h)+\la_{4,1}  h^5 B_5(h)+ \la_{6,1} h^7 B_7(h).$$

\item[{\rm (viii)}] Lotka--Volterra Hamiltonian center: $\la_{1,0}=\la_{4,0}=\la_{5,0}=\la_{3,0}-\la_{6,0}=0$ and $\la_{2,0}\neq 0$
\begin{equation*}\label{eq:LVH1}
\begin{array}{ll}
\dot{x}=-y-\la_{6,0}x^2+2\la_{2,0}xy+\la_{6,0}y^2-\eps \la_{3,1}x^2+\eps^2\la_{5,2}xy+\eps^3\ \la_{1,3}x, \\
\dot{y}=x+\la_{2,0}x^2+2\la_{6,0}xy-\la_{2,0}y^2 +\eps( 2\la_{3,1}
+\la_{4,1})xy+\eps^3\ \la_{1,3}y .
\end{array}
\end{equation*}
The corresponding essential Melnikov function is the third one and it has the form
$$\begin{array}{lll} \displaystyle M_3(h) & = & \displaystyle \la_{1,3} h B_1(h) + \la_{3,1} \la_{5,2} h^3 B_3(h)+\la_{3,1}(\la_{4,1}+5\la_{3,1})\la_{4,1}h^5 B_5(h) \\ & & \displaystyle+ \la_{3,1}^2\la_{4,1}h^7 B_7(h).\end{array} $$

\item[{\rm (ix)}] Symmetric Lotka--Volterra Hamiltonian center (or Hamiltonian
triangle): $\la_{1,0}=\la_{2,0}=\la_{4,0}=\la_{5,0}=\la_{3,0}-\la_{6,0}=0$ and  $\la_{6,0}\neq 0$
\begin{equation*}\label{eq:SLVH}
\begin{array}{ll}
\dot{x}=-y-\la_{6,0}x^2+\la_{6,0}y^2-\eps (x^2+2\la_{2,1}xy)+\eps^3\  \la_{5,3}xy+\eps^4\ \la_{1,4}x,\\
\dot{y}=x+2\la_{6,0}xy +\eps\left( \la_{2,1}x^2+( -2
+\la_{4,1})xy-\la_{2,1}y^2\right)+\eps^4\ \la_{1,4}y .
\end{array}
\end{equation*}
The corresponding essential Melnikov function is the fourth one and it has the form
$$M_4(h)=\la_{1,4} h B_1(h)+\la_{5,3} h^3 B_3(h)+\la_{2,1}(\la_{4,1}+5)\la_{4,1}h^5 B_5(h)+ \la_{2,1}\la_{4,1}h^7 B_7(h).$$

\item[{\rm (x)}] Linear center: $\la_{1,0}=\la_{2,0}=\la_{3,0}=\la_{4,0}=\la_{5,0}=\la_{6,0}=0$
\begin{equation*}\label{eq:lin}
\begin{array}{ll}
\dot{x}=-y+ \eps(-5x^2 + y^2 + 8 \la_{2,1}xy )/4 + \eps^5 \la_{5,5}xy +\eps^6 \la_{1,6}x,\\
\dot{y}=x+\eps(-5xy + 2 \la_{2,1}(x^2-y^2))+\eps^3 \la_{4,3}xy + \eps^6 \la_{1,6} y.
\end{array}
\end{equation*}
The corresponding essential Melnikov function is the sixth one and it has the form
$$M_6(h)=\la_{1,6}h B_1(h)+\la_{5,5} h^3 B_3(h)+\la_{2,1}\la_{4,3} h^5 B_5(h)+ \la_{2,1}(16\la_{2,1}^2-3)h^7 B_7(h).$$

\end{itemize}
\end{theorem}

We consider that the cases {\rm (viii)} and {\rm (ix)} in the above theorem require a discussion. Note that, in the case {\rm (viii)} the coefficients of the Bautin functions which form $M_3(h)$ vary in some set which is dense in $\mathbb{R}^4$, but it is not the whole $\mathbb{R}^4$ (for details one might see the proof of Lemma 5 (viii)). Anyway, for well chosen values of the parameters there are Melnikov functions whose coefficients vary in the complement of the range  of the coefficients of $M_3(h)$. Hence, to study the cyclicity of the period annulus, one can identify the Bautin functions from the expression of $M_3(h)$ and study the zeros of any linear combination of these functions. In the case that the maximum number of zeros (counted with multiplicity) is realized by simple zeros, the cyclicity is found. Otherwise, one can find an upper bound of the cyclicity, but the determination of its exact value, as it is known, is a complicated problem. The same discussion is valid for the case {\rm (ix)}.

\subsection{Essential perturbations of linear centers with cubic nonlinearities}

As it was proved by K.S. Sibirsky \cite{Sib}, see also \cite{Sch} and the references therein, by an affine change of coordinates, any cubic homogeneous center
can be written
\begin{equation} \label{cubic} \begin{array}{lll} \displaystyle \dot{x} & = & \displaystyle - y + \lambda x -
(\omega+\theta-a)x^3-(\eta-3\mu)x^2y \\ & & \displaystyle -
(3\omega-3\theta+2a-\xi)xy^2-(\mu-\nu)y^3,
\\ \displaystyle \dot{y} & = & \displaystyle x + \lambda y +
(\mu+\nu)x^3+(3\omega+3\theta+2a)x^2y \\ & & \displaystyle
+(\eta-3\mu)xy^2+(\omega-\theta-a)y^3,
\end{array}\end{equation} where $\lambda$, $\omega$, $\theta$, $a$, $\eta$, $\mu$, $\xi$,
$\nu$ are real parameters.

It can be shown that this family has the following set of
Poincar\'e--Liapunov constants:
\[ v_1 \, = \, \lambda, \ v_3 \, = \, \xi, \ v_5 \, = \, \nu a, \
v_7 \, = \, \omega \theta a, \ v_9 \, = \, \theta a^2 \eta, \]
\[ v_{11} \, = \, \theta \left[4(\mu^2+\theta^2)-a^2\right] a^2. \]

The center cases of system (\ref{cubic}) are the following:
\begin{itemize}
\item[{\rm (I)}] Hamiltonian: $\lambda\, = \, \xi \, = \, a \, = \, 0$;
\item[{\rm (II)}] Symmetric: $\lambda\, = \, \xi \, = \, \nu \, = \,  \theta  \, = \, 0$;
\item[{\rm (III)}] Darboux: $\lambda \, = \, \xi \, = \, \nu \, = \, \omega \, = \, \eta \, = \, \left[4(\mu^2+\theta^2)-a^2\right]  \, = \, 0$.
\end{itemize}

Analogously to the previous subsection, next lemma provides the expressions of the
coefficients of the Melnikov functions, the essential order and the
essential parameters for all possible positions of a point in the center variety. This lemma is
followed by a theorem which gives the essential perturbation and the
essential Melnikov function in each situation.

\begin{lemma} \label{lemcubic}
For any integer $k>0$, the following statements hold.
\begin{itemize}
\item[{\rm (i)}] Generic Hamiltonian center: $\lambda_0 \, = \, a_0 \, = \, \xi_0 \, = \, 0$ and $\nu_0^2+\theta_0^2 \ne 0$.
\begin{itemize}

\smallskip
\item Case 1: $\nu_0^2+\omega_0^2 \ne 0$.

If $\bar{v}_{1,j} \, = \, \bar{v}_{3,j} \, = \, \bar{v}_{5,j} \, = \, \bar{v}_{7,j} \, = \, 0$ for $j=\overline{0, k-1}$ then
    \[ \begin{array}{lll}
    \bar{v}_{1,k} & = & \lambda_{k}, \\ \bar{v}_{3,k} & = & \xi_{k}, \\ \bar{v}_{5,k} & = & \nu_0 \, a_k, \\ \bar{v}_{7,k} & = & \omega_0 \theta_0 \, a_k, \\ \bar{v}_{9,k} & = & \bar{v}_{11,k}  \, = \, 0.
    \end{array} \]
The essential order is $k^*=1$ and the essential parameters can be chosen to be $\la_1,\xi_1,a_1$.

\smallskip

\item Case 2:  $\nu_0 = \omega_0 = 0$.

Then
\[ \begin{array}{lll} \bar{v}_{1,1} & = & \lambda_{1},  \\ \bar{v}_{3,1} & = & \xi_1, \\ \bar{v}_{5,1} & = & \bar{v}_{7,1} \, = \, \bar{v}_{9,1} \, = \, \bar{v}_{11,1} \, = \, 0. \end{array} \]
    If $\bar{v}_{1,1} \, = \, \bar{v}_{3,1} \, = \, 0$ then
    \[ \begin{array}{lll}
    \bar{v}_{1,2} & = & \lambda_{2}, \\ \bar{v}_{3,2} & = & \xi_{2}, \\ \bar{v}_{5,2} & = & a_{1} \nu_{1}, \\ \bar{v}_{7,2} & = & \theta_0 \, a_{1} \omega_{1}, \\ \bar{v}_{9,2} & = & \eta_0 \theta_0 \, a_{1}^2, \\ \bar{v}_{11,2} & = & 4 \theta_0 (\mu_0^2+ \theta_0^2) \, a_{1}^2.
    \end{array} \]

If $\bar{v}_{1,j} \, = \, \bar{v}_{3,j} \, = \, \bar{v}_{5,j} \, =
\, \bar{v}_{7,j} \, = \, \bar{v}_{9,j} \, = \, \bar{v}_{11,j} \, =
\, 0$ for $j=\overline{1, \ k-1}$ with $k \geq 3$, and
\begin{itemize}
\item[$\bullet$] $k$ is even, then
\[ \begin{array}{lll}
    \bar{v}_{1,k} & = & \lambda_{k}, \\ \bar{v}_{3,k} & = & \xi_{k}, \\ \bar{v}_{5,k} & = & a_{i} \nu_{k-i}, \\ \bar{v}_{7,k} & = & \theta_0 \, a_{i} \omega_{k-i}, \\ \bar{v}_{9,k} & = & \eta_0 \theta_0 \, a_{k-1}^2, \\ \bar{v}_{11,k} & = & 4 \theta_0 (\mu_0^2+ \theta_0^2) \, a_{k-1}^2.
    \end{array} \]
\item[$\bullet$] $k$ is odd, then
\[ \begin{array}{lll}
    \bar{v}_{1,k} & = & \lambda_{k}, \\ \bar{v}_{3,k} & = & \xi_{k}, \\ \bar{v}_{5,k} & = & a_{i} \nu_{k-i}, \\ \bar{v}_{7,k} & = & \theta_0 \, a_{i} \omega_{k-i}, \\ \bar{v}_{9,k} & = & \bar{v}_{11,k} \,  = \, 0.
    \end{array} \]
\end{itemize}
\end{itemize}
The essential order is $k^*=2$ and the essential parameters can be chosen to be $\la_2,\xi_2,a_1,\nu_1,\omega_1$.\\

\item[{\rm (ii)}] Generic Symmetric center: $\lambda_{0} \, = \, \xi_0 \, = \, \nu_0 \, = \, \theta_0 \, = \, 0$, $a_0 \neq 0$ and
$\omega_0^2+\eta_0^2+ (4\mu_0^2-a_0^2)^2 \ne 0$.

If $\bar{v}_{1,j} = \bar{v}_{3,j} = \bar{v}_{5,j} =  \bar{v}_{7,j} =  \bar{v}_{9,j} =  \bar{v}_{11,j} =  0$ for $j=\overline{0,k-1}$, then
    \[ \begin{array}{lll}
    \bar{v}_{1,k} & = & \lambda_{k}, \\ \bar{v}_{3,k} & = & \xi_{k}, \\ \bar{v}_{5,k} & = & a_{0} \nu_{k}, \\ \bar{v}_{7,k} & = & a_0 \omega_0 \, \theta_{k}, \\ \bar{v}_{9,k} & = & a_0^2 \eta_0 \theta_{k}, \\ \bar{v}_{11,k} & = & a_0^2(4\mu_0^2-a_0^2) \theta_k.
    \end{array} \]
The essential order is $k^*=1$ and the essential parameters can be chosen to be $\la_1,\xi_1,\nu_1,\theta_1$.\\

\item[{\rm (iii)}] Generic Darboux center: $\lambda_{0} = \xi_0 = \nu_0 = \omega_0 = \eta_0 = 4(\mu_0^2+\theta_0^2)-a_0^2 \, = \, 0$, $a_0 \ne 0$ and $\theta_0 \ne 0$.

    Then
    \[ \begin{array}{lll}
    \bar{v}_{1,1} & = & \lambda_{1}, \\ \bar{v}_{3,1} & = & \xi_{1}, \\ \bar{v}_{5,1} & = & a_{0} \, \nu_{1}, \\ \bar{v}_{7,1} & = &  a_{0} \, \theta_0 \, \omega_{1}, \\ \bar{v}_{9,1} & = & a_0^2 \theta_{0} \eta_1, \\
    \bar{v}_{11,1} &  = & 2a_0^2 \theta_{0} (4\mu_0\mu_1+4\theta_0\theta_1-a_0a_1).
    \end{array} \]
The essential order is $k^*=1$ and the essential parameters can be chosen to be $\la_1,\xi_1,\nu_1,\omega_1,\eta_1,a_1$.\\

\item[{\rm (iv)}] Hamiltonian symmetric center: $\la_0=\xi_0=a_0=\nu_0=\theta_0=0$ and $\omega_0^2+\eta_0^2+\mu_0^2 \ne 0$.

Then
\[ \begin{array}{lll} \bar{v}_{1,1} & = & \lambda_{1},  \\ \bar{v}_{3,1} & = & \xi_1, \\ \bar{v}_{5,1} & = & \bar{v}_{7,1} \, = \, \bar{v}_{9,1} \, = \, \bar{v}_{11,1} \, = \, 0. \end{array} \]
    If $\bar{v}_{1,1} \, = \, \bar{v}_{3,1} \, = \, 0$
\[ \begin{array}{lll}
    \bar{v}_{1,2} & = & \lambda_{2}, \\ \bar{v}_{3,2} & = & \xi_{2}, \\ \bar{v}_{5,2} & = & a_{1} \, \nu_{1}, \\ \bar{v}_{7,2} & = & \omega_0 \, a_{1}  \theta_{1}, \\ \bar{v}_{9,2} & = & \bar{v}_{11,2} \,  = \, 0.
    \end{array} \]
If $\bar{v}_{1,j} \, = \, \bar{v}_{3,j} \, = \, \bar{v}_{5,j} \, = \, \bar{v}_{7,j} \, = \, \bar{v}_{9,j} \, = \, \bar{v}_{11,j} \, = \, 0$ for $j=\overline{1,k-1}$ with $k \geq 3$, then $\bar{v}_{9,k} = \bar{v}_{11,k}  =  0$.

The essential order is $k^*=2$ and the essential parameters can be chosen to be $\la_2,\xi_2,\nu_1,\theta_1$ taking $a_1=1$.\\

\item[{\rm (v)}] Symmetric Darboux center: $\lambda_{0} = \xi_0 = \nu_0 = \theta_0 = \omega_0=\eta_0=4 \mu_0^2-a_0^2= 0$ and $a_0 \neq 0$.
Then
\[ \begin{array}{lll} \bar{v}_{1,1} & = & \lambda_{1}, \\ \bar{v}_{3,1} & = & \xi_1, \\ \bar{v}_{5,1} & = & a_0 \nu_1, \\ \bar{v}_{7,1} & = & \bar{v}_{9,1} \, = \, \bar{v}_{11,1} \, = \, 0. \end{array} \]
If $\bar{v}_{1,1} \, = \, \bar{v}_{3,1} \, = \, \bar{v}_{5,1} \, = \, 0$ then
\[ \begin{array}{lll}
    \bar{v}_{1,2} & = & \lambda_{2}, \\ \bar{v}_{3,2} & = & \xi_{2}, \\ \bar{v}_{5,2} & = & a_{0} \, \nu_{2}, \\ \bar{v}_{7,2} & = &  a_{0} \, \theta_1 \, \omega_{1}, \\ \bar{v}_{9,2} & = & a_0^2 \theta_{1} \eta_1, \\
    \bar{v}_{11,2} &  = & a_0^2 \theta_{1} (8 \mu_0 \mu_1 - 2 a_0 a_1).
    \end{array} \]

The essential order is $k^*=2$ and the essential parameters can be chosen to be $\la_2,\xi_2,\nu_2,\omega_1,\eta_1,a_1$ taking $\theta_1=1$.\\

\item[{\rm (vi)}] Linear center: $\la_0=\xi_0=a_0=\nu_0 = \theta_0 = \omega_0 = \eta_0  = \mu_0 = 0$.

If $\bar{v}_{1,j} \, = \, \bar{v}_{3,j} \, = \, \bar{v}_{5,j} \, = \, \bar{v}_{7,j} \, = \, \bar{v}_{9,j} \, = \, \bar{v}_{11,j} \, = \, 0$ for $j=0,1,2,3,4$, then
\[ \begin{array}{lll}
    \bar{v}_{1,5} & = & \lambda_{5}, \\ \bar{v}_{3,5} & = & \xi_{5}, \\ \bar{v}_{5,5} & = & a_{1} \, \nu_{4}, \\ \bar{v}_{7,5} & = &  a_{1} \, \theta_1 \, \omega_{3}, \\ \bar{v}_{9,5} & = & a_1^2 \theta_{1} \eta_2, \\
    \bar{v}_{11,5} &  = & a_1^2 \theta_{1} (4(\mu_1^2+\theta_1^2)-a_1^2).
    \end{array} \]
The essential is $k^*=5$ and the essential parameters can be chosen
to be $\la_5, \xi_5,\nu_4,\omega_3,\eta_2,\theta_1$ taking $a_1=1$.

\end{itemize}
\end{lemma}

\begin{theorem}\label{th:ep3}
The essential perturbations and the essential Melnikov functions
are:
\begin{itemize}

\item[{\rm (i)}] Generic Hamiltonian center: $\la_{0}=a_0=\xi_0=0$ and $\nu_0^2+\theta_0^2 \neq 0$

\begin{itemize}

\item Case 1: $\nu_0^2+\omega_0^2 \ne 0$
\begin{equation*}
\begin{array}{lll}
\dot{x}&=&-y-(\omega_0+ \theta_0)x^3 - (\eta_0-3 \mu_0) x^2 y - (3 \omega_0 - 3 \theta_0)x y^2 \\
&& - (\mu_0-\nu_0) y^3  + \e \la_1 x + \e a_1 x^3 - \e (2 a_1 - \xi_1) x y^2,\\

\dot{y}&=&x+ (\mu_0+\nu_0) x^3 + ( 3 \omega_0+ 3 \theta_0) x^2y + (\eta_0-3 \mu _0) x y^2 \\
&& + (\omega_0- \theta_0) y^3 + \e \la_1 y + \e 2 a_1 x^2 y - \e a_1 y^3.
\end{array}
\end{equation*}
The corresponding essential Melnikov function is the first one and it has the form
$$M_1(h)=\la_{1}h B_1(h)+\xi_1h^3 B_3(h)+\nu_0 a_1 h^5 B_5(h)+ \omega_0 \theta_0 a_1 h^7 B_7(h).$$

\item Case 2. $\nu_0=\omega_0=0$
\begin{equation*}
\begin{array}{lll}
\dot{x}&=&-y-\theta_0x^3 - (\eta_0-3 \mu_0) x^2 y +  3 \theta_0 x y^2 - \mu_0 y^3 \\
&& - \e (\omega_1-a_1)x^3 - \e (3 \omega_1+ 2 a_1) x y^2+ \e \nu_1 y^3 + \e^2 \la_2 x + \e^2 \xi_2 x y^2,\\

\dot{y}&=&x+ \mu_0 x^3 + 3 \theta_0 x^2y + (\eta_0-3 \mu _0) x y^2 - \theta_0 y^3\\
&& + \e \nu_1 x^3 + \e ( 3 \omega_1 + 2 a_1 ) x^2 y - \e a_1 y^3 + \e^2 \la_2 y.
\end{array}
\end{equation*}
The corresponding essential Melnikov function is the second one and it has the form
$$ \begin{array}{lll} \displaystyle M_2(h) & = & \displaystyle \la_{2}hB_1(h)+\xi_2h^3 B_3(h)+a_1 \nu_1 h^5 B_5(h)+ a_1 \omega_1 h^7 B_7(h) \\ & & \displaystyle + a_1^2(\eta_0  h^9 B_9(h)
+  h^{11} B_{11}(h)).\end{array} $$
\end{itemize}

\item[{\rm (ii)}] Generic symmetric center: $\lambda_{0} \, = \, \xi_0 \, = \, \nu_0 \, = \, \theta_0 \, = \, 0$, $a_0 \neq 0$ and
$\omega_0^2+\eta_0^2+ (4\mu_0^2-a_0^2)^2 \ne 0$.

\begin{equation*} \begin{array}{lll} \displaystyle \dot{x} & = & \displaystyle - y  -
(\omega_0-a_0)x^3-(\eta_0-3\mu_0)x^2y   -
(3\omega_0+2a_0)xy^2-\mu_0y^3\\ & & \displaystyle + \e\lambda_1 x-\e
\theta_1x^3  + \e(3\theta_1+\xi_1)xy^2+\e\nu_1y^3,
\\ \displaystyle \dot{y} & = & \displaystyle x  +
\mu_0x^3+(3\omega_0+2a_0)x^2y
+(\eta_0-3\mu_0)xy^2+(\omega_0-a_0)y^3\\ & & \displaystyle +
\e\lambda_1 y+ \e\nu_1x^3+\e3\theta_1x^2y -\e\theta_1y^3 .
\end{array}\end{equation*}
The corresponding essential Melnikov function is the first one and
it has the form
$$ \begin{array}{lll} \displaystyle M_1(h) & = & \displaystyle \la_{1}hB_1(h)+\xi_1h^3 B_3(h)+\nu_1 h^5 B_5(h)+ \theta_1[\omega_0  h^7 B_7(h)+\eta_0h^9 B_9(h)\\ & & \displaystyle +(4\mu_0^2-a_0^2)h^{11} B_{11}(h)]. \end{array}$$

\item[{\rm (iii)}] Generic Darboux center: $\lambda_{0} = \xi_0 = \nu_0 = \omega_0 = \eta_0 = 4(\mu_0^2+\theta_0^2)-a_0^2 \, = \, 0$, $a_0 \ne 0$ and $\theta_0 \ne 0$.

\begin{equation*} \begin{array}{lll} \displaystyle \dot{x} & = & \displaystyle - y  -
(\theta_0-a_0)x^3+3\mu_0x^2y  - (-3\theta_0+2a_0)xy^2-\mu_0y^3\\ & &
\displaystyle+ \e\lambda_1 x- \e(\omega_1-a_1)x^3-\e\eta_1x^2y  -\e
(3\omega_1+2a_1-\xi_1)xy^2-\nu_1y^3,
\\ \displaystyle \dot{y} & = & \displaystyle x  +
\mu_0x^3+(3\theta_0+2a_0)x^2y -3\mu_0xy^2-(\theta_0+a_0)y^3 \\ & &
\displaystyle+ \e\lambda_1 y+ \e\nu_1x^3+\e(3\omega_1+2a_1)x^2y
+\e\eta_1xy^2+\e(\omega_1-a_1)y^3.
\end{array}\end{equation*}
The corresponding essential Melnikov function is the first one and
it has the form
$$ \begin{array}{lll} \displaystyle M_1(h) & = & \displaystyle \la_{1}hB_1(h)+\xi_1h^3 B_3(h)+\nu_1 h^5 B_5(h)+ \omega_1  h^7 B_7(h)+\eta_1h^9 B_9(h)\\ & & \displaystyle +a_1h^{11} B_{11}(h).\end{array} $$

\item[{\rm (iv)}] Symmetric Hamiltonian center: $\la_0=\xi_0=a_0=\nu_0=\theta_0=0$ and $\omega_0^2+\eta_0^2+\mu_0^2 \ne 0$.

\begin{equation*} \begin{array}{lll} \displaystyle \dot{x} & = & \displaystyle - y  -
\omega_0x^3-(\eta_0-3\mu_0)x^2y  - 3\omega_0xy^2-\mu_0y^3\\ & &
\displaystyle + \e (1-\theta_1)x^3  +
\e(3\theta_1-2)xy^2+\e\nu_1y^3+\e^2 \lambda_2 x +\e^2\xi_2xy^2,
\\ \displaystyle \dot{y} & = & \displaystyle x  +
\mu_0x^3+3\omega_0x^2y +(\eta_0-3\mu_0)xy^2+\omega_0y^3 \\ & &
\displaystyle + \e\nu_1x^3+\e(3\theta_1+2)x^2y -\e(1+\theta_1)y^3+
\e^2\lambda_2 y.
\end{array}\end{equation*}
The corresponding essential Melnikov function is the second one and
it has the form
$$M_2(h)=\la_{2}hB_1(h)+\xi_2h^3 B_3(h)+\nu_1 h^5 B_5(h)+  \omega_0\theta_1 h^7 B_7(h).$$

\item[{\rm (v)}] Symmetric Darboux center: $\lambda_{0} = \xi_0 = \nu_0 = \theta_0 = \omega_0=\eta_0=4 \mu_0^2-a_0^2= 0$ and $a_0 \neq 0$.

\begin{equation*} \begin{array}{lll} \displaystyle \dot{x} & = & \displaystyle - y  +a_0x^3+3\mu_0x^2y  -
2a_0xy^2-\mu_0y^3  - \e(\omega_1-a_1+1)x^3  \\ & & \displaystyle    -\e
\eta_1x^2y  - \e(3\omega_1+2a_1-3)xy^2 + \e^2\lambda_2 x
+\e^2\xi_2xy^2+\e^2\nu_2y^3 ,
\\ \displaystyle \dot{y} & = & \displaystyle x  +
\mu_0x^3+2a_0x^2y -3\mu_0xy^2-a_0y^3 +\e(3\omega_1+2a_1+3)x^2y
\\ & & \displaystyle
+\e\eta_1xy^2+\e(\omega_1-a_1-1)y^3+\e^2
\lambda_2 y +\e^2\nu_2x^3.
\end{array}\end{equation*}
The corresponding essential Melnikov function is the second one and
it has the form
$$ \begin{array}{lll} \displaystyle M_2(h) & = & \displaystyle \la_{2}hB_1(h)+\xi_2h^3 B_3(h)+ \nu_2 h^5 B_5(h)+  \omega_1 h^7 B_7(h)+ \eta_1 h^9 B_9(h) \\ & & \displaystyle + a_1 h^{11} B_{11}(h). \end{array}$$

\item[{\rm (vi)}] Linear center: $\la_0=\xi_0=a_0=\nu_0 = \theta_0 = \omega_0 = \eta_0  = \mu_0 = 0$.

\begin{equation*} \begin{array}{lll} \displaystyle \dot{x} & = & \displaystyle - y  -
\e(\omega_1-1)x^3+\e (3\theta_1-2)xy^2-\e^2\eta_2x^2y\\ & &
\displaystyle-\e^3\omega_3x^3-3\e^3\omega_3xy^2+
  \e^4\nu_4y^3+ \e^5\lambda_5 x+\e^5\xi_5xy^2,
\\ \displaystyle \dot{y} & = & \displaystyle x  +\e(3\theta_1+2)x^2y
-\e(\theta_1+1)y^3 +\e^2\eta_2xy^2\\ & &
\displaystyle+3\e^3\omega_3x^2y+\e^3\omega_3y^3+
  \e^4\nu_4x^3+ \e^5\lambda_5  y.
\end{array}\end{equation*}
The corresponding essential Melnikov function is the fifth one and
it has the form
$$ \begin{array}{lll} \displaystyle M_5(h) & = & \displaystyle \la_{5}hB_1(h)+\xi_5h^3 B_3(h)+ \nu_4 h^5 B_5(h)+ \theta_1\omega_3 h^7 B_7(h) \\ & & \displaystyle +\theta_1 \eta_2 h^9 B_9(h) + (4\theta_1^3-\theta_1) h^{11} B_{11}(h). \end{array}$$
\end{itemize}
\end{theorem}

\section{Example \label{sect4}}
Consider the following  system with a center at the origin
\begin{equation}\label{ex}
\dot{x} \, = \, -y(1+y), \quad \dot{y} \, = x(1+y),
\end{equation}
having the first integral $H(x,y)=\sqrt{x^2+y^2}$ and the corresponding
inverse integrating factor  $V(x,y)=(1+y)\sqrt{x^2+y^2}$. Its period annulus is
$\mathcal{P} \, = \, \{ H=h\ :\  h \in (0,1)\}$. System \eqref{ex}
is in the standard Bautin form and, according to the classification
of quadratic centers (given in paragraph 3.1), is a generic
symmetric (reversible) center. Consider now a perturbation of system
(\ref{ex}) by quadratic polynomials with coefficients which are
analytic in the small bifurcation parameter $\eps$:
\begin{equation}\label{exp} \dot{x} \, = \, -y(1+y)+\varepsilon p(x,y, \varepsilon), \  \dot{y} \, = x(1+y)+\varepsilon q(x,y, \varepsilon). \end{equation}
As we explained in the beginning of paragraph 3.1, there exists an
affine change of variables which is analytic with respect to $\eps$
that transforms system \eqref{exp} in the Bautin standard form
\eqref{eq:qsB0}. This transformation is the identity for $\eps=0$,
in this case, meaning that the unperturbed system \eqref{ex} does
not change after this transformation. Note that we have
$(\la_{1,0},\la_{2,0},\la_{3,0},\la_{4,0},\la_{5,0},\la_{6,0})=(0,0,0,1,0,-1)$.
We apply  Theorem \ref{th:ep2} (ii) and deduce that an essential perturbation
of center \eqref{ex} is
\begin{equation*}
\begin{array}{ll}
\dot{x}=-y(1+y)+\eps\left(  \la_{1,1}x+\left(2\la_{2,1}+\la_{5,1}\right)xy\right)\\
\dot{y}=x(1+y)+ \eps \left( \la_{1,1}y+\la_{2,1}x^2
-\la_{2,1}y^2\right),
\end{array}
\end{equation*}
and the essential Melnikov function is the first one and it has the
form
$$M_1(h)=\la_{1,1}hB_1(h)+\la_{5,1}h^3 B_3(h)+\la_{2,1}  h^5\tilde B_5(h). $$
As it is proved in \cite{BGY}, there are at most $2$ zeroes of $M_1(h)$ in
$\mathcal{P}$. \par Indeed, as it is proved in \cite{Iliev}, when $M_1(h) \equiv 0$, the expression of the higher-order
Melnikov function is analogous to $M_1(h)$. Therefore, the cyclicity of $\mathcal{P}$ under quadratic perturbations
is $2$. However in \cite{BGY} it is stated erroneously that the function $M_3(h)$ can have $3$ zeroes. We remark that, if one uses the perturbative system considered in \cite{BGY} and applies the method described in this manuscript, the same conclusion that the essential Melnikov function is the first one is accomplished.

\section{On the finiteness of the number of limit cycles bifurcating from the period annulus $\mathcal{P}$ \label{sect5}}

In this manuscript we describe a method to give an essential perturbation for a family of planar polynomials differential systems (\ref{eq2}) which unfold a system with a period annulus $\mathcal{P}$ corresponding to a nondegenerate center. \par The existence of a essential perturbation may induce the idea that the cyclicity of any period annulus $\mathcal{P}$ surrounding a nondegenerate center is finite. Assume that, for a particular family
(\ref{eq3}), we have that $M_{k^*}(h)$ is the essential Melnikov function where $k^*$ is the essential order. This implies that if a
particular system (\ref{eq3}) has $\ell$ limit cycles which bifurcate from the orbits of $\mathcal{P}$, then $M_{k^*}(h)$ has at
least $\ell$ isolated zeroes (counted with multiplicity). We recall that $M_{k^*}(h)$ is analytic in the interval $[h_0,h_1)$, where
$h_0 \in \mathbb{R}$ corresponds to the inner boundary, that is the center singular point, and $h_1 \in \mathbb{R} \cup \{+\infty\}$ is
the level set of the first integral $H(x,y)$ corresponding to the outer boundary of $\mathcal{P}$. If the number of isolated zeroes of
$M_{k^*}(h)$ is finite, then the cyclicity of $\mathcal{P}$ is finite. \par
Due to analyticity, any Melnikov function (and in particular the essential Melnikov
function $M_{k^*}$) can have a countable set of zeros. Theoretically
the set of zeros can be both finite and infinite. If the number of
isolated zeroes of $M_{k^*}(h)$ is infinite, then they need to
accumulate to $h_1$ (we remind that $h_1$ is the level value of the
first integral corresponding to the outer boundary of the period
annulus). The fact that this oscillatory behavior does not appear
for a period annulus  of a Hamiltonian or a generic Darboux
integrable system has been shown in \cite{GavNov}; see also the
references therein.\par We remark that the fact that the number of
isolated zeroes of $M_{k^*}(h)$ is infinite does not contradict the
finiteness of the number of limit cycles for a particular fixed
system (\ref{eq3}) which was proved by \'Ecalle \cite{Ecalle} and
Ilyashenko \cite{Ilyashenko}, as we explain below. Assume, to fix ideas, that
$M_{k^*}(h)$ has an infinite number of simple zeroes which we denote
by $\xi_n$, with $n \in \mathbb{N}$. We can assume without loss of
generality that $\xi_n < \xi_{n+1}$ and we have that $\lim_{n\to
\infty} \xi_n \, = \, h_1$. For each $\xi_n$, we have by the
Implicit Function Theorem (see also Theorem \ref{thbifur}) that there exists a value
$\varepsilon_n>0$ and a function $\vartheta_n(\varepsilon)$ analytic
in the interval $\varepsilon \in (-\varepsilon_n, \varepsilon_n)$
such that $d(\vartheta_n(\varepsilon); \varepsilon) \equiv 0$ for
all $|\varepsilon|<\varepsilon_n$. For a fixed value $\varepsilon\in
(-\varepsilon_n, \varepsilon_n)\setminus \{0\}$, the point $\vartheta_n(\varepsilon)$ corresponds to a limit cycle of the
system (\ref{eq3-e}) which has bifurcated from the periodic orbit
corresponding to the level $\xi_n$. For a fixed value of
$\varepsilon$, system (\ref{eq3-e}) has a finite number of limit
cycles, which implies that $\lim_{n \to \infty} \varepsilon_n \to
0$. Then, given a fixed value of $\varepsilon \neq 0$ there is a
finite number of intervals in the set $\{ (-\varepsilon_n,
\varepsilon_n) \, : \, n \in \mathbb{N} \}$ in which $\varepsilon$
belongs to. Thus, the functions $\vartheta_n(\varepsilon)$ only
exist for this finite number of intervals and, as a consequence,
there is only a finite number of limit cycles for system (\ref{eq3-e})
for the considered fixed value of $\varepsilon$. If we take a value
of $\varepsilon$ closer to $0$ we may have more limit cycles and
since $\varepsilon_n>0$ for all $n \in \mathbb{N}$ and $\lim_{n \to
\infty} \varepsilon_n \to 0$, we have that given a number $\ell$,
there is always a value of $\varepsilon$ close enough to $0$ such
that the corresponding system (\ref{eq3-e}) has at least $\ell$ limit
cycles bifurcating from the periodic orbits of $\mathcal{P}$.
Therefore, even though for a fixed system (\ref{eq3-e}) the number of
limit cycles is finite, one has that the cyclicity of the period
annulus $\mathcal{P}$ is infinite. \par However, it turns out that, as far as
the authors know, there is no example of a Melnikov function with
such an oscillatory behavior. Indeed, in all the papers known by the
authors, the Melnikov function satisfies a Chebyshev property. More
precisely, as we have proved in Theorem \ref{thcj}, $M_k(h)$ can be
written as the linear combination (\ref{eqMkj}) of $N+1$ linearly
independent functions $h^{2j+1}B_{2j+1}(h)$ (called Bautin
functions), which are analytic for $h$ in the whole period annulus
and with $B_{2j+1}(0)$ a nonzero constant, for $j=\overline{0, N}$.
It turns out, in the studied examples, that the Bautin functions are
not only Chebyshev in a neighborhood of the origin but in the whole
period annulus. This implies that the number of isolated zeroes
(counted with multiplicity) of any linear combination of these $N+1$
functions is at most $N$. We recall that given $N+1$ analytic functions on a real interval $L$, they form an extended Chebyshev system (in short ET-system) on $L$ if any nontrivial linear combination has at most $N$ isolated zeros on $L$, counted with multiplicity. Some papers even conjecture such Chebyshev property for some particular systems, see \cite{MarSaaUriWall}.

\section*{Acknowledgements}

The authors are partially supported by a MICINN/FEDER grant number MTM 2011-22877 and by a AGAUR (Generalitat de
Ca\-ta\-lu\-nya) grant number 2009SGR 381. The first author was also partially supported by a grant of the
Romanian National Authority for Scientific Research, CNCS UEFISCDI,
project number PN-II-ID-PCE-2011-3-0094.

%%%%%%%%%%%%%%%%%%%%%%%%%%%%%%%%%%%%%%%%%%%%%%%%%%%%%%%%%%%%%%%%%

\end{document}